\documentclass[a4paper,11pt]{article}
\usepackage{inputenc}
\usepackage{authblk}
\usepackage[table]{xcolor}
\usepackage{color}
\usepackage{amsmath}
\usepackage{amssymb}
\usepackage{amsthm}
\usepackage{booktabs}
\usepackage{xcolor,graphicx,float}
\usepackage{hyperref}
\usepackage{soul}
\usepackage{indentfirst}
\hypersetup{colorlinks=true,
	linkcolor=blue,
	anchorcolor=blue,
	urlcolor=black,
	citecolor=blue}
\usepackage{enumitem}
\setlist[itemize]{itemsep=0pt, topsep=0pt}
\usepackage{tikz}
\usepackage{subcaption}
\usepackage{caption}
 
\definecolor{yalv}{RGB}{150,194,78}
\definecolor{tonglv}{RGB}{43,174,133}
\usepackage[
a4paper,
textwidth=16cm,
textheight=23cm,
centering
]{geometry}

\theoremstyle{plain}
\newtheorem{theorem}{\bf Theorem}[section]
\newtheorem{lemma}[theorem]{\bf Lemma}

\newtheorem{corollary}[theorem]{\bf Corollary}

\theoremstyle{remark}

\newtheorem{problem}{\bf Problem}
\newtheorem{construction}{\bf Construction}
\newtheorem{claim}{\bf Claim}

\newtheorem{assumption}{\bf Assumption}

\numberwithin{equation}{section}

\title{\bf Structure of large $t$-intersecting families I: Stability for the Hilton--Milner--Frankl theorem}
\author[1]{Jie Wen\thanks{E-mail: \text{jwen@mail.bnu.edu.cn}}}
\author[1]{Benjian Lv\thanks{Corresponding author. E-mail: \text{bjlv@bnu.edu.cn}}}
\affil[1]{\small Laboratory of Mathematics and Complex Systems (Ministry of Education), School of Mathematical Sciences, Beijing Normal University, Beijing 100875, China}
\date{}

\begin{document}
	\renewcommand{\baselinestretch}{1.2}
	\maketitle
	\begin{abstract}
		\medskip
We study the structure of large $t$-intersecting
families. A family of $k$-subsets of an $n$-set is $t$-intersecting if every two of its members intersect in at least $t$ elements. A $t$-intersecting family is non-trivial if no $t$-subset is contained in all its members. We prove several stability results for the seminal Hilton--Milner--Frankl theorem. First, for any fixed $\eta,\theta\in(0,1)$, we prove that if $k/t\geq1+\eta$ and $n=\Omega(tk\log k)$, then every non-trivial $t$-intersecting family of size greater than $(1+\theta)|\mathcal{K}|$ is a subfamily of one of the two extremal families in the theorem, where $\mathcal{K}$ is an explicit large non-trivial $t$-intersecting family. The key ingredient in the proof is a removal lemma. We also obtain a classification of all $t$-intersecting families with size bounded below by $|\mathcal{K}|$ minus an explicit lower-order term, provided that $k\geq t+4\geq6$ and $n\geq t+6\cdot\max\{(t+2)^2, k(k-t)\}$. This strengthens results of Cao--Lv--Wang (2021) and Frankl (2025) for a broad range of $k$ and $t$ (for example, when $k-t\geq2\sqrt{t}$). As an application of this classification, we determine the largest $t$-intersecting families for each prescribed lower bound on $t$-diversity not exceeding $t(n-k)$, thereby obtaining $t$-intersection versions of results of Han and Kohayakawa (2017) and Kupavskii (2025).  To establish these results, we develop techniques based on the spread approximation method and the $t$-cover method, which may be useful for other intersection problems.

		
		Spread approximation;\;Stability;\;$t$-cover;\;$t$-diversity
		
	\end{abstract}
	\section{Introduction}
We use $[n]=\{1,2,\ldots,n\}$ to denote the standard $n$-set, and set $[i,j]=\{i,i+1,\ldots,j\}$ for $i\leq j$. A \emph{family} is just a collection of sets. Denote by $\binom{[n]}{k}$ the family of all $k$-subsets of $[n]$. A family $\mathcal{F}$ is \emph{$t$-intersecting} if $|F\cap F'|\geq t$ for all $F,F'\in\mathcal{F}$. Two families $\mathcal{G}_1,\mathcal{G}_2$ of subsets of $[n]$ are \emph{isomorphic}, denoted $\mathcal{G}_1\cong\mathcal{G}_2$, if one can be obtained from the other by a permutation of $[n]$. We write $\mathcal{F}\lesssim\mathcal{G}$ if $\mathcal{F}$ is isomorphic to a subfamily of $\mathcal{G}$. Otherwise, we write $\mathcal{F}\not\lesssim\mathcal{G}$.

The celebrated Erd\H{o}s--Ko--Rado theorem determines maximum-sized $t$-intersecting families in $\binom{[n]}{k}$. The study of intersection problems has since developed into a significant area of extremal combinatorics. For a systematic introduction to the field, we refer the reader to survey papers \cite{Frankl-Tokushige-2016, Ellis-book}. Let us present the following \emph{exact Erd\H{o}s--Ko--Rado theorem}.
\begin{theorem}[\cite{Erdos-Ko-Rado-1961,Frankl-1976,Wilson-1984}]\label{EKR}
	Let $k\geq t\geq 1$ and $n\geq(t+1)(k-t+1)$. If $\mathcal{F}\subseteq\binom{[n]}{k}$ is $t$-intersecting, then
	$$|\mathcal{F}|\leq\binom{n-t}{k-t}.$$
	Moreover, if $n>(t+1)(k-t+1)$, then equality holds if and only if $\mathcal{F}\cong\left\{F\in\binom{[n]}{k}:[t]\subseteq F\right\}$.
\end{theorem}	
 Denote by $n_0(k,t)$ the least possible value for $n$ such that the upper bound above holds. Determining or estimating $n_0(k,t)$ was a major open problem. The original paper \cite{Erdos-Ko-Rado-1961}, published in 1961, proved $n_0(k,1)=2k$, and established $n_0(k,t)\leq(k-t)\binom{k}{t}^3+t$. In 1976, Frankl \cite{Frankl-1976} made a breakthrough by determining $n_0(k,t)=(t+1)(k-t+1)$ for $t\geq 15$, and proving $n_0(k,t)\leq2(t+1)(k-t+1)$ for all $t$. In 1984,  Wilson \cite{Wilson-1984} resolved the problem via an ingenious algebraic method. In 1997, Ahlswede and Khachatrian \cite{Ahlswede-Khachatrian-1997} established the famous \emph{complete intersection theorem}, which determines largest $t$-intersecting families in $\binom{[n]}{k}$ for all $n,k$ and $t$. This also settled a conjecture of Frankl \cite{Frankl-1976}.
 
Besides finding largest $t$-intersecting families, another longstanding and active direction of research concerns the  problem of characterizing the structure of large $t$-intersecting families. A $t$-intersecting family $\mathcal{F}\subseteq\binom{[n]}{k}$ is \emph{trivial} if $|\cap_{F\in\mathcal{F}}F|\geq t$, and is \emph{non-trivial} otherwise. The family $\mathcal{F}$ is \emph{maximal} if $\mathcal{F}\cup\{A\}$ is not $t$-intersecting for any $A\in\binom{[n]}{k}\setminus\mathcal{F}$. A maximal  trivial $t$-intersecting family in $\binom{[n]}{k}$ is said to be a \emph{full $t$-star}. Note that for $n>n_0(k,t)$, the full $t$-star centered at $[t]$ is, up to isomorphism, the unique  family achieving the maximum size. For $n\leq n_0(k,t)$, the family
\begin{equation}\label{familya}
	\mathcal{A}(n,k,t):=\left\{F\in\binom{[n]}{k}:|F\cap [t+2]|\geq t+1\right\},
\end{equation}
which is known as a \emph{Frankl family} (see e.g.,  \cite{Frankl-Tokushige-2016}), has size not less than that of a full $t$-star. Another important example of a  non-trivial $t$-intersecting family is
\begin{equation}\label{familyh}
	\mathcal{H}(n,k,t):=\left\{F\in\binom{[n]}{k}:[t]\subseteq F,\;F\cap[t+1,k+1]\neq\emptyset\right\}\cup\left\{[k+1]\setminus\{i\}:i\in [t]\right\}.
\end{equation}
We next recall the famous Hilton--Milner--Frankl theorem.
\begin{theorem}[\cite{Hilton-Milner-1967,Frankl-1978}]\label{HM}
	Let $k\geq t\geq 1$ and $n>n_1(k,t)$. If $\mathcal{F}\subseteq\binom{[n]}{k}$ is a non-trivial $t$-intersecting family, then
	$$|\mathcal{F}|\leq\max\{|\mathcal{A}(n,k,t)|,  |\mathcal{H}(n,k,t)|\}.$$
	Moreover, if equality holds, then
\begin{itemize}
	\item[\rm(i)] $\mathcal{F}\cong\mathcal{H}(n,k,t)$ for $k\geq2t+2$, or
	\item[\rm(ii)] $\mathcal{F}\cong\mathcal{A}(n,k,t)$ for $k\leq 2t+1$, or $(k,t)=(3,1)$ and $\mathcal{F}\cong\mathcal{H}(n,3,1)$.
\end{itemize}	
\end{theorem}
In 1967, Hilton and Milner \cite{Hilton-Milner-1967} proved Theorem \ref{HM} for $t=1$ and $n_1(k,1)=2k$, which is best possible. Frankl \cite{Frankl-1978} obtained the result for $t\geq2$. Ahlswede and Khachatrian \cite{Ahlswede-Khachatrian-1996} proved that the upper bound in Theorem \ref{HM} holds for all $n\geq(t+1)(k-t+1)$, and further completely determined largest non-trivial $t$-intersecting families in $\binom{[n]}{k}$ for all parameters.

 A result on the structure of large $t$-intersecting families often contributes to the \emph{stability} for the Erd\H{o}s--Ko--Rado theorem or the Hilton--Milner--Frankl theorem. More precisely, such a result implies that a large $t$-intersecting family which is nearly extremal in size is also close to the optimal families in structure. One classical form of stability is a \emph{removal theorem}, which asserts that a nearly optimal family can be made a subfamily of an extremal construction by deleting relatively few members. The following removal theorem was proved by Ellis, Keller and Lifshitz \cite{Ellis-2019}. 
 \begin{theorem}[\cite{Ellis-2019}]\label{thmEKL}
 	Let $n,k,t\in\mathbb{N}$ and $\eta>0$ with $\eta n\leq k\leq(\frac{1}{t+1}-\eta)n$, let $\varepsilon>0$, and let $\mathcal{F}\subseteq\binom{[n]}{k}$ be $t$-intersecting with $|\mathcal{F}|\geq(1-\varepsilon)\binom{n-t}{k-t}$. Then there exists a full $t$-star $\mathcal{S}$ such that
 	\begin{equation*}
 		|\mathcal{F}\setminus\mathcal{S}|\leq O_{t,\eta}(\varepsilon^{\log_{1-k/n}(k/n)})\binom{n-t}{k-t}.
 	\end{equation*}
 \end{theorem}
 
 We refer the reader to \cite{Dinur-Friedgut,Keevash,Keevash-Long-2020,Ellis-2024} for more results on removal-type stability for intersection theorems. To introduce  subsequent results, let us define the following construction.
 \begin{construction}
 	For $1\leq c\leq k-t$, define
 	\begin{align*}
 		\mathcal{K}(n,k,t,c)=&\left\{F\in\binom{[n]}{k}:[t]\subseteq F,\;F\cap[t+1,k]\neq\emptyset\right\}\\
 		&\cup\left\{F\in{[n]\choose k}:F\cap[k]=[t],\;[k+1,k+c]\subseteq F\right\}\\
 		&\cup\left\{([k]\setminus\{i\})\cup\{j\}:i\in[t],\;j\in[k+1,k+c]\right\}.
 	\end{align*}  
 Note that $\mathcal{H}(n,k,t)=\mathcal{K}(n,k,t,1)$.
 \end{construction}

 Our first main result is a stability result for the Hilton--Milner--Frankl theorem. The order of the required lower bound on $n$ is asymptotically optimal up to a logarithmic factor.
\begin{theorem}\label{corothmremoval}
 	For any $\eta,\theta\in(0,1)$, there exists a constant $C=C(\eta,\theta)$ such that the following holds. Let $k\geq t+2$, $k\geq(1+\eta)t$ and $n\geq Ctk\log_2k$. If $\mathcal F\subseteq\binom{[n]}k$ is a non-trivial $t$-intersecting family with $|\mathcal F|>(1+\theta)|\mathcal{K}(n,k,t,2)|$, then $\mathcal{F}\lesssim\mathcal{A}(n,k,t)$ or $\mathcal{F}\lesssim\mathcal{H}(n,k,t)$.
 \end{theorem}
Theorem \ref{corothmremoval} follows almost immediately from the following removal theorem, which is the main technical result of this paper. It shows that a family near the threshold $|\mathcal{K}(n,k,t,2)|$ can be made a subfamily of a copy of $\mathcal{K}(n,k,t,2)$ by deleting a small fraction of members.
 
  \begin{theorem}\label{thmremoval}
 	For any $\eta,\theta\in(0,1)$, there exists a constant $C=C(\eta,\theta)$ such that the following holds. Let $k\geq t+2$, $k\geq(1+\eta)t$ and $n\geq Ctk\log_2k$. If $\mathcal F\subseteq\binom{[n]}k$ is a non-trivial $t$-intersecting family with $\mathcal{F}\not\lesssim\mathcal{A}(n,k,t)$, $\mathcal{F}\not\lesssim\mathcal{H}(n,k,t)$ and $|\mathcal F|\geq(1-\theta)|\mathcal{K}(n,k,t,2)|$, then there exists $\mathcal{K}\subseteq\binom{[n]}{k}$ with $\mathcal{K}\cong\mathcal{K}(n,k,t,2)$ such that  $|\mathcal{F}\setminus\mathcal{K}|\leq2\theta|\mathcal{K}(n,k,t,2)|.$
 \end{theorem}

 Our approach is based on the \emph{peeling procedure} introduced by Kupavskii and Zakharov \cite{Kupavskii-2024}, which is a key ingredient of their \emph{spread approximation method}. The method has proved surprisingly powerful, leading to several  breakthroughs in extremal set theory \cite{Kupavskii-2023,Kupavskii-2024-perm,Frankl-Kupavskii-2025,Kupavskii-2026,Keller-Kupavskii-Lifshitz-Sheinfeld}.  A $q$-analog of the peeling procedure was recently developed by
 Ihringer and Kupavskii \cite{Ihringer-Kupavskii}. We also rely crucially on the notion of a \emph{$t$-cover} (which will be defined below) of a $t$-intersecting family, developed by Cao, the second author and Wang \cite{Lv-2021}. The $t$-cover method can also be regarded as a strengthening of the classical ``kernel method'' (see e.g. \cite{H-R}) and has proved useful in characterizing large $t$-intersecting families (see e.g., \cite{Cao-Lu-Lv-Wang-2024,Yao-Liu-Wang,Wen-Lv-2026+}). To prove our results, we incorporate the structural information
provided by $t$-covers into the peeling procedure. This yields
considerably sharper estimates, valid for smaller values of $n$,
while retaining the information needed to characterize the families
under consideration. 

A notable feature of the proof of Theorem \ref{thmremoval} is that it naturally
 splits into two regimes according to the size of $n$. When
 $n\geq C_1k^2$ for a suitable constant $C_1$, we apply the
 peeling procedure directly to $\mathcal{F}$, and use the structural information encoded by its minimal $t$-covers. In the complementary range $Ctk\log_2 k\leq n<C_1k^2$, we first construct a spread approximation $\mathcal{S}$ of $\mathcal{F}$, and then apply the
 peeling procedure to $\mathcal S$. A key ingredient in the latter argument is Theorem \ref{thmstrongsp}, which we prove using the spread
 approximation method. It yields a $t$-intersecting family
 $\mathcal S$ whose members have size $O(t\log_2k)$, such that all but a suitably small number of members of $\mathcal{F}$ contain some $S\in\mathcal{S}$. 

We next turn to exact structural results for large $t$-intersecting families. Inspired by the Hilton--Milner theorem, Han and Kohayakawa \cite{Han-Kohayakawa} determined the maximum size of a non-trivial $1$-intersecting family $\mathcal{F}\subseteq\binom{[n]}{k}$ with $\mathcal{F}\not\lesssim\mathcal{H}(n,k,1)$, that is, not  isomorphic to a subfamily of $\mathcal{H}(n,k,1)$. More precisely, for $n>2k$ and $k\geq3$, they proved that either  $|\mathcal{F}|\leq|\mathcal{K}(n,k,1,2)|$, or $k=3$ and $|\mathcal{F}|\leq|\mathcal{A}(n,k,1)|$, and they also characterized the extremal configurations. Further, Kostochka and Mubayi  \cite{Kostochka-Mubayi} described the structure of $1$-intersecting families with size at least $|\mathcal{K}(n,k,1,k-1)|$ for all sufficiently large $n$ depending on $k$. This has been established for $n>2k$ by Kupavskii \cite{Kupavskii-2018,Kupavskii-2025}, and by Huang and Peng \cite{Huang-Peng}. For more results on this question, see \cite{Ge-Xu-Zhao,Wu et al.,Huang-Kupavskii-2026}. For general $t\geq1$, Cao, the second author and Wang \cite{Lv-2021} determined, for large $n$, all maximal $t$-intersecting families of size at
 least $|\mathcal K(n,k,t,k-t)|$, as well as those whose size is
 slightly below this threshold. Let us note that, since every $t$-intersecting family is contained in some maximal one, it is  sufficient to characterize the structure of maximal families. To present their result, let us define the following construction of families.
\begin{construction}
For $2\leq u\leq k-t+1$, define
	\begin{align*}
		\mathcal{L}(n,k,t,u)=&\left\{F\in\binom{[n]}{k}:[t]\subseteq F,\;F\cap[t+1,t+u]\neq\emptyset\right\}\\
		&\cup\left\{F\in\binom{[n]}{k}:F\cap[t+u]=[t+u]\setminus\{i\}
		\text{ for some } i\in[t]\right\}.
	\end{align*} 
\end{construction}
Note that $\mathcal{A}(n,k,t)=\mathcal{L}(n,k,t,2)$ and $\mathcal{H}(n,k,t)=\mathcal{L}(n,k,t,k-t+1)$. By a routine counting argument, we have
\begin{align}
|\mathcal{K}(n,k,t,c)|&=\binom{n-t}{k-t}-\binom{n-k}{k-t}+\binom{n-k-c}{k-t-c}+ct.\\
|\mathcal{L}(n,k,t,u)|&=\binom{n-t}{k-t}-\binom{n-u-t}{k-t}+t\binom{n-u-t}{k-u-t+1}.
\end{align}
 \begin{theorem}[\cite{Lv-2021}]\label{thmCLW}
 	Let $k\geq t+2$ and $n\geq t+\max\left\{{t+2\choose 2},\frac{k-t+2}{2}\right\}\cdot(k-t+1)^2$. If  $\mathcal{F}\subseteq{[n]\choose k}$ is a maximal non-trivial $t$-intersecting family with $
 	|\mathcal{F}|\geq (k-t){n-t-1\choose k-t-1}-{k-t\choose 2}{n-t-2\choose k-t-2}$, then one of the following holds.		
 	\begin{itemize}
 		\item[\rm(i)] $\mathcal{F}\cong\mathcal{L}(n,k,t,s)$ for some $s\in\{2,k-t,k-t+1\}$.
 		\item[\rm(ii)] $\mathcal{F}\cong\mathcal{K}(n,k,t,c)$ for some $c\in\{2,3,\ldots,k-t\}$.
 	\end{itemize}		
 \end{theorem}
  Roughly speaking, results in extremal set theory are usually much easier to establish provided that the ground set $[n]$ is sufficiently large. More recently, Frankl \cite{Frankl-2025} considered the same structural problem under the condition $n\geq t+\max\{4t(k-t+1)^2,2(t+1)^2\}$. Our next main result gives the corresponding characterization under an improved bound for a broad range of $k$ and $t$. 
 \begin{theorem}\label{thmlarge}
 	Let $k\geq t+4\geq6$ and $n\geq t+6\cdot\max\{(t+2)^2, k(k-t)\}$. If $\mathcal{F}\subseteq{[n]\choose k}$ is a maximal non-trivial $t$-intersecting family with $|\mathcal{F}|\geq\binom{n-t}{k-t}-\binom{n-k}{k-t}$, then either 
 	\begin{itemize}
 		\item[\rm(i)] $\mathcal{F}\cong\mathcal{L}(n,k,t,s)$ for some $s\in\{2,k-t,k-t+1\}$, or
 		\item[\rm(ii)] $\mathcal{F}\cong\mathcal{K}(n,k,t,c)$ for some $c\in\{2,3,\ldots,k-t\}$.
 	\end{itemize}
 \end{theorem}
 Let us note that our method works for $k-t\in\{2,3\}$ as well. However, handling these two cases requires
 separate and rather tedious estimates. We therefore
 restrict the statement of Theorem \ref{thmlarge} to $k\geq t+4$. For $k-t\in\{2,3\}$, Theorem
 \ref{thmCLW} already provides a reasonably strong bound. We do not attempt to optimize the factor $6$. Together with a simple counting lemma (Lemma \ref{lemmacountingfamily}),
 Theorem \ref{thmlarge} yields a $t$-intersecting analogue of the Han--Kohayakawa theorem. 
 \begin{corollary}\label{coroH-K}
Let $k\geq t+4\geq6$ and $n\geq t+6\cdot\max\{(t+2)^2, k(k-t)\}$. If $\mathcal{F}\subseteq{[n]\choose k}$ is non-trivial $t$-intersecting with
$\mathcal F\not\lesssim\mathcal A(n,k,t)$ and
$\mathcal F\not\lesssim\mathcal H(n,k,t)$, then 
$$|\mathcal{F}|\leq\binom{n-t}{k-t}-\binom{n-k}{k-t}+\binom{n-k-2}{k-t-2}+2t,$$
with equality if and only if $\mathcal{F}\cong\mathcal{K}(n,k,t,2)$.
 \end{corollary}
 
We now turn to $t$-diversity and another application of Theorem
\ref{thmlarge}. For a family $\mathcal{F}$ of subsets of $[n]$, its \emph{$s$-degree}, denoted $\Delta_s(\mathcal{F})$, is defined to be the maximum number of its members that contain a fixed $s$-subset. Formally,
 \begin{equation}\label{equdefDelta}
 \Delta_s(\mathcal{F}):=\max_{S\in\binom{[n]}{s}}|\left\{F\in\mathcal{F}:S\subseteq F\right\}|.
 \end{equation}
We also define its \emph{$t$-diversity} by
 \begin{equation*}
 	\gamma_t(\mathcal{F}):=|\mathcal{F}|-\Delta_t(\mathcal{F}).
 \end{equation*}
Let us note that the notion of $t$-diversity is a generalization of the important parameter \emph{diversity} (see e.g., \cite{Lemons-Palmer,Frankl-2020,Frankl-Kupavskii-2021}). When $t=1$, the two notions coincide. We write $\Delta$ and $\gamma$ simply for $\Delta_t$ and $\gamma_t$, respectively. One can regard $\gamma_t(\mathcal{F})$ as the distance of $\mathcal{F}$ to its closest full $t$-star. In this sense, a $t$-intersecting family is non-trivial precisely if it has positive $t$-diversity. For the two constructions defined above, it is routine to check that
$$\gamma_t(\mathcal{K}(n,k,t,c))=ct\;\;\;\mbox{and}\;\;\gamma_t(\mathcal{L}(n,k,t,u))=t\binom{n-u-t}{k-u-t+1}$$
for $c\in[k-t]$ and $u\in[2,k-t+1]$. 

In 1987, Frankl \cite{Frankl-1987} proved a far-reaching strengthening of the Hilton--Milner theorem, characterizing the largest $1$-intersecting families with bounded maximum degree. Let us present its diversity version established by Kupavskii and Zakharov \cite{Kupavskii-Zakharov-2018}.
\begin{theorem}[\cite{Kupavskii-Zakharov-2018}]\label{thmK-Z}
Let $n>2k$ and $\mathcal{F}\subseteq\binom{[n]}{k}$ be a $1$-intersecting family. If $\gamma(\mathcal{F})\geq\binom{n-u-1}{n-k-1}$ for some \emph{real} $3\leq u\leq k$, then 
$$|\mathcal{F}|\leq\binom{n-1}{k-1}-\binom{n-u-1}{k-1}+\binom{n-u-1}{n-k-1}.$$
\end{theorem}
 Let us note that, for each integer $u\in[3,k]$, the theorem above proves that, if $\mathcal{F}\subseteq\binom{[n]}{k}$ is $1$-intersecting with $\gamma(\mathcal{F})\geq\gamma(\mathcal{L}(n,k,1,u))$, then $|\mathcal{F}|\leq|\mathcal{L}(n,k,1,u)|$. In particular, the case $u=k$ yields the upper bound in the Hilton--Milner theorem.
 
 In \cite{Kupavskii-2025}, Kupavskii determined, for every $1\leq m\leq n-k$, the structure of maximum-sized $1$-intersecting families with diversity at least $m$. This
 completes the classification throughout the range between  $\gamma(\mathcal{L}(n,k,1,k))=1$ and $\gamma(\mathcal{L}(n,k,1,k-1))=n-k$. For general $t\geq2$, the authors established in \cite{Wen-Lv-2026+} the following $t$-intersection version of Theorem \ref{thmK-Z} (for integer values of $u$) for $n=\Omega(k^3)$. More precisely, let $k\geq t+2\geq4$, $n\geq7(k-t+1)k^2$ and $u\in[2,k-t]$. If $\mathcal{F}\subseteq\binom{[n]}{k}$ is $t$-intersecting with $\gamma_t(\mathcal{F})\geq\gamma_t(\mathcal{L}(n,k,t,u))$, then either $\mathcal{F}\lesssim\mathcal{A}(n,k,t)$, or $|\mathcal{F}|\leq|\mathcal{L}(n,k,t,u)|$. As a consequence of Theorem \ref{thmlarge}, we obtain the following strengthening of the Hilton--Milner--Frankl theorem, which determines largest $t$-intersecting
families for every prescribed lower bound $m$ on the
$t$-diversity in the range $m\leq\gamma_t(\mathcal{L}(n,k,t,k-t))=t(n-k)$.
 
\begin{theorem}\label{thmdiv}
Let $k\geq t+4\geq6$ and $n\geq t+6\cdot\max\{(t+2)^2, k(k-t)\}$, and let $m$ be an integer with $1\leq m\leq t(n-k)$. If $\mathcal{F}\subseteq{[n]\choose k}$ is a $t$-intersecting family with maximum size under the condition $\gamma_t(\mathcal{F})\geq m$, then either $\mathcal{F}\cong\mathcal{A}(n,k,t)$, or the following holds, where $c=\lceil m/t\rceil$. 
\begin{itemize}
\item[\rm(i)]$\mathcal{F}\cong\mathcal{H}(n,k,t)$ for $c=1$.
\item[\rm(ii)]$\mathcal{F}\cong\mathcal{K}(n,k,t,c)$ for $2\leq c\leq k-t-2$.
\item[\rm(iii)]$\mathcal{F}\cong\mathcal{L}(n,k,t,k-t)$ for $c\geq k-t-1$.
\end{itemize}
\end{theorem}

The rest of this paper is organized as follows. In the next section, we introduce the peeling procedure and a crucial technique (Lemma \ref{lemmamaximal}) for applying it to maximal $t$-intersecting families, and collect several important properties. In Section \ref{secfin-stru}, we first prove three structural lemmas (Lemmas \ref{lemmafin-stru0}, \ref{lemmafin-stru1-single} and \ref{lemmafin-stru2}), which provide quantitative estimates and structural characterizations in terms of the families produced by the procedure. Then we prove Theorems \ref{thmlarge} and \ref{thmdiv}. In Section \ref{secspreadapproximation}, we establish Theorem \ref{thmstrongsp} utilizing the spread approximation method, and then prove our removal theorem (Theorem \ref{thmremoval}), from which Theorem \ref{corothmremoval} follows almost immediately. Finally, Section \ref{secremark} contains some concluding remarks.
\section{Preliminaries}\label{secpeeling}
For families $\mathcal{A}$ and $\mathcal{S}$ of subsets of $[n]$ and a subset $X$ of $[n]$, we write
\begin{align*}
	\mathcal{A}[X]&:=\{A:A\in\mathcal{A},X\subseteq A\}.\\
	\mathcal{A}[\overline{X}]&:=\{A:A\in\mathcal{A},X\nsubseteq A\}.\\
	\mathcal{A}(X)&:=\{A\setminus X:A\in\mathcal{A},X\subseteq A\}.\\
	\mathcal{A}[\mathcal{S}]&:=\left\{A:A\in\mathcal{A}, S\subseteq A\;\mbox{for some}\;S\in\mathcal{S}\right\}.
\end{align*}
Note that $\mathcal{A}(X)$ has the same size as $\mathcal{A}[X]$, while it is a family of subsets of $[n]\setminus X$. Recall from (\ref{equdefDelta}) that the maximum $s$-degree of a family $\mathcal{A}$ is defined to be $\Delta_s(\mathcal{A})=\max_{S\in\binom{[n]}{s}}|\mathcal{A}[S]|$. A family $\mathcal{A}$ is said to be \emph{$r$-spread} if 
$$\Delta_s(\mathcal{A})\leq r^{-s}|\mathcal{A}|\;\;\mbox{for}\;\;s=0,1,2,\ldots$$
Spreadness plays an important role in a series of recent major advances in combinatorics (see e.g. \cite{sunflower,Kupavskii-2024}). The following is a useful inductive property.
\begin{lemma}[\cite{Kupavskii-2024}]\label{lemmaspreadness-size}
	If $\mathcal{F}\subseteq\binom{[n]}{k}$ and $|\mathcal{F}|>r^k$ for some $r>1$, then there is a subset $X$ of $[n]$ with size smaller than $k$ such that $\mathcal{F}(X)$ is $r$-spread.
\end{lemma}
For a family $\mathcal{F}$ of subsets of $[n]$, a  \emph{$t$-cover} of $\mathcal{F}$ is a set that intersects every member of $\mathcal{F}$ in at least $t$ elements. If such a set exists, then the \emph{$t$-covering number} of $\mathcal{F}$, denoted $\tau_t(\mathcal{F})$, is defined to be the minimum size of a $t$-cover. In particular, if $\mathcal{F}\subseteq\binom{[n]}{k}$ is $t$-intersecting, then every member of $\mathcal{F}$ itself forms a $t$-cover, and so $\tau_t(\mathcal{F})\leq k$. A $t$-cover $T$ of $\mathcal{F}$ is \emph{minimal} if none of its proper subset is a $t$-cover, that is, for all $T_0\subsetneqq T$, there exists $F\in\mathcal{F}$ such that $|F\cap T_0|<t$. Of course a $t$-cover of size $\tau_t(\mathcal{F})$ is minimal.

Let us introduce the notion of a fingerprint for  $t$-intersecting families.
\begin{lemma}\label{lemmafingerprintdef}
	Let \(\mathcal S\) be a
	\(t\)-intersecting family of subsets of $[n]$. Then there exists a
	\(t\)-intersecting family \(\mathcal S^*\), called a \emph{fingerprint} of
	\(\mathcal S\), such that
	\begin{itemize}
		\item[\rm(i)] every member of \(\mathcal S^*\) is contained in some member
		of \(\mathcal S\);
		\item[\rm(ii)] every member of \(\mathcal S\) contains some member of
		\(\mathcal S^*\);
		\item[\rm(iii)] every member of \(\mathcal S^*\) is a minimal \(t\)-cover
		of \(\mathcal S^*\).
	\end{itemize}
\end{lemma}
\begin{proof}
One can obtain such a family by repeatedly replacing members of $\mathcal{S}$ by proper subsets that preserve the $t$-intersecting property. Indeed, write $u=|\mathcal{S}|$ for short. Fix an ordering $(S_1,S_2,\ldots,S_u)$ of members of $\mathcal{S}$. We replace each $S_i$ ($i=1,2,\ldots,u$) by an inclusion-minimal subset $S_i^*\subseteq S_i$ that $t$-intersects every set in the sequence, that is, $S_i^*$ intersects every set from $(S_1^*,\ldots,S_{i-1}^*,S_i,S_{i+1},\ldots,S_u)$ in at least $t$ elements, and the property does not hold for any $T\subsetneqq S_i^*$. Then one desired family is obtained by deleting repeated sets in the resulting sequences $(S_1^*,S_2^*,\ldots,S_u^*)$.
\end{proof}

We formulate the peeling procedure as the following algorithm. As mentioned, it was introduced by Kupavskii and Zakharov in \cite{Kupavskii-2024}. Roughly speaking, given a $t$-intersecting family together with an `initial' fingerprint, the procedure outputs a sequence of families encoding structural information on the original family. For simplicity, we use $\binom{[n]}{\leq q}$ to denote the family of subsets of $[n]$ with size at most $q$.\vspace{1em}

\noindent{\bf Algorithm.\label{algo-single}\;}
Suppose that integers \(q\) and \(t\) with
\(q\ge t\), a \(t\)-intersecting family
\(\mathcal S\subseteq\binom{[n]}{\le q}\), and a fingerprint
\(\mathcal S_0\) of \(\mathcal S\) are given. Set \(N=0\). For
\(i=0,1,\ldots,q-t\), do the following:
\begin{itemize}
	\item[(a)] Set $
	\mathcal X_i=\mathcal S_i\cap\binom{[n]}{q-i}.$
	If \(\mathcal S_i=\mathcal X_i\), terminate the procedure and set
	\(N=i\).
	
	\item[(b)] Find a fingerprint of $
	\mathcal S_i\setminus\mathcal X_i$, 
	and denote it by \(\mathcal S_{i+1}\).
\end{itemize}

\noindent Output \(N\) and \((\mathcal S_i,\mathcal X_i)\) for
\(i\le N\).\vspace{1em}

The key to applying this method is to start with a suitably chosen initial fingerprint. Let $\mathcal{F}\subseteq\binom{[n]}{k}$ be $t$-intersecting. We define
\begin{align}
	\mathcal{M}(\mathcal{F})&:=\{T\subseteq[n]:T\;\mbox{is a minimal $t$-cover of $\mathcal{F}, |T|\leq k$}\}.\\
	\mathcal{M}^*(\mathcal{F})&:=\{T\in\mathcal{M}(\mathcal{F}):|T|=\tau_t(\mathcal{F})\}.\label{equmstardef}
\end{align}
The following key lemma suggests that, when $\mathcal{F}$ is maximal, we may start with $\mathcal{M}(\mathcal{F})$. 
\begin{lemma}\label{lemmamaximal}
	Suppose $n\geq 2k$ and $\mathcal{F}\subseteq\binom{[n]}{k}$ is a maximal $t$-intersecting family. Then $\mathcal{M}(\mathcal{F})$ is a fingerprint of $\mathcal{F}$. In particular,
	$$\mathcal{F}=\left\{F\in\binom{[n]}{k}:M\subseteq F\;\mbox{for some}\;M\in\mathcal{M}(\mathcal{F})\right\}.$$
\end{lemma}
\begin{proof}
First, let us prove that $\mathcal{M}(\mathcal{F})$	is $t$-intersecting. Let $S_1,S_2\in\mathcal{M}(\mathcal{F})$. Since $|S_1|,|S_2|\leq k$ and $n\geq2k$, there are $F_1,F_2\in\binom{[n]}{k}$ with $S_1\subseteq F_1$, $S_2\subseteq F_2$ and $F_1\cap F_2=S_1\cap S_2$. Note that both $S_1$ and $S_2$ are $t$-covers of $\mathcal{F}$. Then the maximality of $\mathcal{F}$ yields $F_1,F_2\in\mathcal{F}$, and hence $|S_1\cap S_2|=|F_1\cap F_2|\geq t$. Thus $\mathcal{M}(\mathcal{F})$ is $t$-intersecting. 

Next, we verify (i)-(iii) in Lemma \ref{lemmafingerprintdef} for $\mathcal{S}=\mathcal{F}$ and   $\mathcal{S}^*=\mathcal{M}(\mathcal{F})$. For each $S\in\mathcal{M}(\mathcal{F})$, of course $S$ is a $t$-cover of $\mathcal{F}$, and it follows from the maximality of $\mathcal{F}$ that $\mathcal{F}$ contains all $k$-subsets of $[n]$ containing $S$. Then certainly $S$ is contained in some $F\in\mathcal{F}$, and thus (i) holds. For (ii), conversely, every member of $\mathcal{F}$ is a $t$-cover of $\mathcal{F}$ itself, and so one of its subsets belongs to $\mathcal{M}(\mathcal{F})$. It remains to prove that every $S\in\mathcal{M}(\mathcal{F})$ is a minimal $t$-cover of $\mathcal{M}(\mathcal{F})$. For such a set $S$, and for an arbitrary proper subset $S_0$ of $S$, we see that $|F_0\cap S_0|<t$ for some $F_0\in\mathcal{F}$, because $S$ is by definition a minimal $t$-cover of $\mathcal{F}$. Using the property (ii), there exists $S'\subseteq F_0$ with $S'\in\mathcal{M}(\mathcal{F})$. Then clearly $|S_0\cap S'|<t$, implying that $S_0$ is not a $t$-cover of $\mathcal{M}(\mathcal{F})$. Therefore, $S$ is minimal. 
\end{proof}

The following lemma collects several important properties for the output families in the peeling procedure. Parts (i)--(iii) are standard; the sharper estimate in (iv) adapts a counting argument from \cite{Kupavskii-2024-perm}. Let us recall an elementary but useful inequality. 
\begin{equation}\label{equltrt}
	\binom{a-c}{b-c}\leq(a-b+1)^{b-c}\;\;\mbox{for}\;\;a\geq b\geq c.
\end{equation}
This can be easily verified as  $\frac{y}{x}<\frac{y-1}{x-1}$ for $y>x>1$. For simplicity, we define an expression.
\begin{align*}
\varphi(u,t):=\sum_{j=0}^{t}\binom{t}{j}\binom{u}{j}^2(u+1)^{u-j}.
\end{align*}
\begin{lemma}\label{lemmafingerprintproperty-single}
Perform the algorithm to
$\mathcal{S}$, the uniformity $q$ and an initial fingerprint 
$\mathcal{S}_0$ of $\mathcal{S}$. Let $N$ be the terminal index and $(\mathcal{S}_i,\mathcal{X}_i)$ ($i\leq N$) be the output
families. For $i=0,1,\ldots,N$, the following hold.
		\item[\rm(i)] $\mathcal{S}_i\subseteq\binom{[n]}{\leq(q-i)}$ and $\mathcal{X}_{i}\subseteq\binom{[n]}{q-i}$, and both are antichains.
\item[\rm(ii)]$\mathcal{S}[\mathcal{S}_{i-1}]\subseteq\mathcal{S}[\mathcal{X}_{i-1}]\cup\mathcal{S}[\mathcal{S}_i]$ and  $\mathcal{S}\subseteq\left(\bigcup_{j=0}^{i-1}\mathcal{S}[\mathcal{X}_j]\right)\cup\mathcal{S}[\mathcal{S}_i]$ for $i\geq1$. 
\item[\rm(iii)]If $\mathcal{H}$ is a $(q-a)$-uniform subfamily of $\mathcal{S}_i$, then for each $B\subseteq[n]$, it holds that $$|\mathcal{H}[B]|\leq(q-t-i+1)^{q-a-|B|}.$$
\item[\rm(iv)]$|\mathcal{X}_i|\leq\varphi(q-t-i,t)$.
\end{lemma}
\begin{proof}
(i) By the choice of $\mathcal{S}_i$ and $\mathcal{X}_i$, it is clear that
$\mathcal{S}_i\subseteq\binom{[n]}{\leq q-i}$ and
$\mathcal{X}_i\subseteq\binom{[n]}{q-i}$. To see that $\mathcal{S}_i$ is
an antichain, just note that every member of $\mathcal{S}_i$ is a minimal
$t$-cover of $\mathcal{S}_i$. Hence no member of $\mathcal{S}_i$ is
contained in another one. The same holds for $\mathcal{X}_i$ as
$\mathcal{X}_i\subseteq\mathcal{S}_i$.

(ii) Let $i\geq1$. The first inclusion follows directly from
$\mathcal{S}_{i-1}=\mathcal{X}_{i-1}\cup
(\mathcal{S}_{i-1}\setminus\mathcal{X}_{i-1})$ and the fact that every
set in $\mathcal{S}_{i-1}\setminus\mathcal{X}_{i-1}$ contains some member
of $\mathcal{S}_i$. Thus $
\mathcal{S}[\mathcal{S}_{i-1}]
\subseteq
\mathcal{S}[\mathcal{X}_{i-1}]\cup\mathcal{S}[\mathcal{S}_i]$. The second inclusion follows by using the first one repeatedly, together
with $\mathcal{S}\subseteq\mathcal{S}[\mathcal{S}_0]$.

(iii)\;
We proceed by proving a claim.
\begin{claim}\label{claimlemmafingerprintproperty-single}
For each $i\leq N$, if $X$ is a subset of $[n]$ and $\mathcal{L}$ is a subfamily of $\mathcal{S}_i$ such that $\mathcal{L}(X)$ is $r$-spread and $|\mathcal{L}(X)|\geq2$, then $r\leq q-t-i+1$.
\end{claim}
To prove the claim, note that $|\mathcal{L}(X)|\geq2$ implies that set $X$ is a proper subset of some
member of $\mathcal{S}_i$. As every member of $\mathcal{S}_i$ is a minimal
$t$-cover of $\mathcal{S}_i$, there exists $S\in\mathcal{S}_i$ with
$|X\cap S|=:s<t$. Since $\mathcal{L}\subseteq\mathcal{S}_i$ and
$\mathcal{S}_i$ is $t$-intersecting, every member of $\mathcal{L}$ has
intersection at least $t$ with $S$. Hence
$$
\mathcal{L}[X]
=
\bigcup_{Z\in\binom{S\setminus X}{t-s}}\mathcal{L}[X\cup Z].
$$
By picking a $(t-s)$-subset $Z_0$ of $S\setminus X$ maximizing
$|\mathcal{L}[X\cup Z_0]|$, we obtain that
$$
|\mathcal{L}[X]|
\leq
\binom{|S|-s}{t-s}|\mathcal{L}[X\cup Z_0]|
\leq
(q-t-i+1)^{t-s}\cdot r^{-(t-s)}|\mathcal{L}[X]|.
$$
Here we used (\ref{equltrt}), $|S|\leq q-i$ and the $r$-spreadness of $\mathcal{L}(X)$. Thus $r\leq q-t-i+1$, as claimed. 

Let us prove (iii). To the contrary, assume that
$|\mathcal{H}[B]|>(q-t-i+1)^{q-b-a}$ for some $B\in\binom{[n]}{b}$.
Then there is an $r>q-t-i+1\geq1$ such that
$|\mathcal{H}[B]|>r^{q-b-a}$. By Lemma \ref{lemmaspreadness-size}, there
exists a subset $W$ of $[n]\setminus B$ with $|W|<q-b-a$ such that
$\mathcal{H}(B\cup W)=(\mathcal{H}(B))(W)$ is $r$-spread. Since
$|B\cup W|<q-a$, the set $B\cup W$ is a proper subset of some
$F\in\mathcal{H}[B\cup W]$. Therefore
$$
1=|(\mathcal{H}(B\cup W))(F\setminus(B\cup W))|
\leq
r^{-(|F|-|B\cup W|)}|\mathcal{H}(B\cup W)|
\leq
r^{-1}|\mathcal{H}(B\cup W)|.
$$
It follows that $|\mathcal{H}(B\cup W)|\geq\lceil r\rceil\geq2$. This
contradicts Claim \ref{claimlemmafingerprintproperty-single} by identifying $\mathcal{L}$ with $\mathcal{H}$ and $X$
with $B\cup W$. Hence $
|\mathcal{H}[B]|\leq(q-t-i+1)^{q-b-a}.
$

(iv)\;First, we observe that $|A\cap B|=t$ for some $A,B\in\mathcal{S}_i$. Indeed, fix an $A\in\mathcal{S}_i$, then any $(|A|-1)$-subset of $A$ is not a $t$-cover of $\mathcal{S}_i$, since $A$ is minimal. Then some $B\in\mathcal{S}_{i}$ intersects $A$ in exactly $t$ elements. For $0\leq j\leq t$, denote by $\mathcal{D}_j$ the set of triples $(X,Y,Z)$ with $X\in\binom{A\cap B}{t-j}, Y\in\binom{A\setminus B}{j}$ and $Z\in\binom{B\setminus A}{j}$. Every member of $\mathcal{X}_i$ must $t$-intersect both $A$ and $B$, and hence contains $X\cup Y\cup Z$ for some $(X,Y,Z)\in\cup_j\mathcal{D}_j$. Therefore,
\begin{equation*}
\mathcal{X}_i=\bigcup_{0\leq j\leq t}\left(\bigcup_{(X,Y,Z)\in\mathcal{D}_j}\mathcal{X}_i[X\cup Y\cup Z]\right).
\end{equation*}
By the union bound, and by applying (iii) with $\mathcal{H}=\mathcal{X}_i$ and $a=i$, we obtain
\begin{align*}
|\mathcal{X}_i|&\leq\sum_{j=0}^{t}\sum_{(X,Y,Z)\in\mathcal{D}_j}|\mathcal{X}_i[X\cup Y\cup Z]|\\
&\leq\sum_{j=0}^{t}\binom{t}{t-j}\binom{|A\setminus B|}{j}\binom{|B\setminus A|}{j}(q-t-i+1)^{(q-i)-(t+j)}\\
&\leq\sum_{j=0}^{t}\binom{t}{j}\binom{q-i-t}{j}^2(q-t-i+1)^{q-t-i-j}=\varphi(q-t-i,t),
\end{align*}
as desired.
\end{proof}
\section{Characterizing structures via fingerprints}\label{secfin-stru}
In this section, we establish several structural lemmas (Lemmas \ref{lemmafin-stru0}, \ref{lemmafin-stru1-single} and \ref{lemmafin-stru2}), and then give a classification of $t$-intersecting families with size at least $\binom{n-t}{k-t}-\binom{n-k}{k-t}$. To begin with, let us collect several technical estimates.
\begin{lemma}[\cite{Wen-Lv-2026+}]\label{lemmabinom}
	Let $m\geq u+v$. The following hold.
	\begin{itemize}
		\item[\rm(i)]$(1+uv/m)\binom{m-v}{u}\leq\binom{m}{u}\leq(1+m/(uv))\left(\binom{m}{u}-\binom{m-v}{u}\right).$
		\item[\rm(ii)]$\binom{m}{u}-\binom{m-v}{u}\leq v\binom{m-1}{u-1}$ and $\binom{m-v}{u}\geq(1-uv/m)\binom{m}{u}.$
		\item[\rm(iii)]$\binom{m}{u}-\binom{m-v}{u}\geq v\binom{m-1}{u-1}-\binom{v}{2}\binom{m-2}{u-2}$.
	\end{itemize} 
\end{lemma}
\begin{lemma}\label{lemmacountingfamily}
Suppose $k\geq t+3\geq5$ and $n\geq2k+t$. Then
\begin{align*}
	|\mathcal{K}(n,k,t,1)|&>|\mathcal{K}(n,k,t,2)|>\cdots>|\mathcal{K}(n,k,t,k-t-2)|\\
	&>|\mathcal{L}(n,k,t,k-t)|>|\mathcal{K}(n,k,t,k-t-1)|>|\mathcal{K}(n,k,t,k-t)|.
\end{align*}
\end{lemma}
\begin{proof}
It is routine to check that  $|\mathcal{K}(n,k,t,c)|=\binom{n-t}{k-t}-\binom{n-k}{k-t}+\binom{n-k-c}{k-t-c}+ct$ for $2\leq c\leq k-t$, and $|\mathcal{L}(n,k,t,k-t)|=\binom{n-t}{k-t}-\binom{n-k}{k-t}+t(n-k)$. So 
	$$|\mathcal{K}(n,k,t,c)|-|\mathcal{K}(n,k,t,c+1)|=\binom{n-k-c-1}{k-t-c}-t\geq n-k-c-1-t>0$$
	for $1\leq c\leq k-t-1$. Next, we have 
	$$|\mathcal{K}(n,k,t,c)|-|\mathcal{L}(n,k,t,k-t)|=\binom{n-k-c}{k-t-c}-t(n-k-c).$$
	Then $|\mathcal{K}(n,k,t,k-t-2)|>|\mathcal{L}(n,k,t,k-t)|>|\mathcal{K}(n,k,t,k-t-1)|$.
\end{proof}
Let us recall 
\begin{align}
	\varphi(u,t)=\sum_{j=0}^{t}\binom{t}{j}\binom{u}{j}^2(u+1)^{u-j}.\label{equfunphi}
\end{align}
\begin{lemma}\label{lemmarqt}
	Suppose $0<\varepsilon<1$ and $r,q,t\geq1$ with $q\geq t+1$ and  $\varepsilon r\geq eq$. For $1\leq u\leq q-t$, set $T_{u}=\binom{u+t}{t}(u+1)^u/r^u$. For $t+1\leq m\leq q$, set two expressions as follows.
\begin{align}
	\Phi(r,q,t,m)&=\sum_{i=0}^{q-m}\varphi(q-t-i,t)/r^{q-t-i}.\\
	\Psi(r,q,t,m)&=\sum_{j=m-t}^{q-t}T_j=\sum_{i=0}^{q-m}\binom{q-i}{t}\left(\frac{q-t-i+1}{r}\right)^{q-t-i}.\label{equfunPsi}
\end{align}
The following hold.
\begin{itemize}
\item[\rm(i)]$\varphi(u,t)\leq r^uT_u$, and $\Phi(r,q,t,m)\leq\varphi(m-t,t)/r^{m-t}+\Psi(r,q,t,m+1)\leq\Psi(r,q,t,m)$.
\item[\rm(ii)]$T_{j+1}\leq\varepsilon T_j$ for $1\leq j\leq q-t-1$, and $\Psi(r,q,t,m)<T_{m-t}/(1-\varepsilon).$
\end{itemize}
\end{lemma}
\begin{proof}
(i)\;By the Vandermonde identity, we have
\begin{align*}
\varphi(u,t)&=\sum_{j=0}^{t}\binom{t}{j}\binom{u}{j}^2(u+1)^{u-j}\\&\leq(u+1)^u\sum_{j=0}^{t}\binom{t}{j}\binom{u}{j}=(u+1)^u\binom{t+u}{u}=r^uT_u.
\end{align*}
To see the inequality, just note that the expression $\binom{u}{j}(u+1)^{u-j}$ is decreasing on $j$. Hence
\begin{align*}
\Phi(r,q,t,m)&=\sum_{j=m-t}^{q-t}\varphi(j,t)/r^{j}\leq\varphi(m-t,t)/r^{m-t}+\sum_{j=m-t+1}^{q-t}T_j\\
&=\varphi(m-t,t)/r^{m-t}+\Psi(r,q,t,m+1)\\&\leq T_{m-t}+\Psi(r,q,t,m+1)=\Psi(r,q,t,m).
\end{align*}

(ii)\;Since $\varepsilon r\geq eq$, we have  $$\frac{T_{j+1}}{T_j}=\frac{\binom{j+t+1}{t}(j+2)^{j+1}}{r\binom{j+t}{t}(j+1)^j}=\frac{j+t+1}{r}\cdot\left(1+\frac{1}{j+1}\right)^{j+1}<\frac{q}{r}\cdot e\leq\varepsilon$$
	for $1\leq j\leq q-t-1$. Then
	\begin{align*}
		\Psi(r,q,t,m)&=\sum_{j=m-t}^{q-t}T_j<T_{m-t}(1+\varepsilon+\cdots+\varepsilon^{q-m})<T_{m-t}/(1-\varepsilon),
	\end{align*}
as desired.
\end{proof}
\begin{lemma}[\cite{Wen-Lv-2026+}]\label{lemmaind}
	Let $n\geq(k-t+1)(\ell-t+1)+t$ and $\mathcal{F}\subseteq\binom{[n]}{k}$. Suppose that $G\in\binom{[n]}{\ell}$ and $A\in\binom{[n]}{a}$, where $\ell\ge t$ and $t\le a\le k$. If $G$ is a $t$-cover of $\mathcal{F}$ and $|A\cap G|=s<t$, then 
	\begin{equation*}
		|\mathcal{F}[A]|\leq\binom{n-a}{k-a}-\binom{n-a-(\ell-t+1)}{k-a}\leq(\ell-t+1)\binom{n-a-1}{k-a-1}.
	\end{equation*}
\end{lemma}
The following lemma serves as a key estimate obtained from the $t$-cover method.  
\begin{lemma}[\cite{Wen-Lv-2026+}]\label{lemmakey}
	Let $t\leq m\leq k$, $n\geq(k-t+1)(\ell-t+1)+t$ and $\mathcal{F}\subseteq\binom{[n]}{k}$. Suppose that $\mathcal{G}$ is a collection of 
	$t$-covers of $\mathcal{F}$, where each has size at most $\ell$. If $\tau_t(\mathcal{G})\geq m$, then  $$|\mathcal{F}[X]|\leq(\ell-t+1)^{m-t}\binom{n-m}{k-m}$$
	for all $X\in\binom{[n]}{t}$. In particular, $|\mathcal{F}|\leq(\ell-t+1)^{m-t}\binom{\tau_t(\mathcal{F})}{t}\binom{n-m}{k-m}.$
\end{lemma}
We now turn to the main estimates and structural lemmas. We first record the standing assumptions and notation used throughout the remainder of this section. 
\begin{assumption}\label{assumption}
Let $k\geq t+2$, $b=k-t+1$, $r=(n-t)/(k-t)$ and
$\varepsilon\in(0,0.5]$. Assume that $\varepsilon r=ek$. Put $\delta_{-2}=\delta_{-1}=0$, and  $$\delta_{k-m}=\varphi(m-t,t)/r^{m-t-1}+rT_{m-t+1}/(1-\varepsilon)\;\;\mbox{for}\;\;t+2\leq m\leq k,$$
 where $T_{u}=\binom{u+t}{t}(u+1)^u/r^u$ for $2\leq u\leq k-t$, $T_{k-t+1}=0$ and $\varphi$ is defined in (\ref{equfunphi}). Write $\alpha=b/r$ and  $\delta=\delta_{k-t-2}$. 
Suppose that $\mathcal{F}\subseteq\binom{[n]}{k}$ is a maximal
non-trivial $t$-intersecting family. Apply the algorithm to
$\mathcal{F}$ with $q:=k$ and the initial fingerprint 
$\mathcal{S}_0:=\mathcal{M}(\mathcal{F})$, and let $N$ be the terminal index and $(\mathcal{S}_i,\mathcal{X}_i)$ ($i\leq N$) be the output
families. Set
$\mathcal{S}^*=\mathcal{S}_0\cap\binom{[n]}{t+1}$. 
\end{assumption}

With the notation above, and using Lemma \ref{lemmafingerprintproperty-single} (ii), we decompose $\mathcal{F}$ via the output families:
\begin{equation*}
\mathcal{F}=\left(\bigcup_{i=0}^{u-1}\mathcal{F}[\mathcal{X}_i]\right)\cup\mathcal{F}[\mathcal{S}_{u}],\;u=1,2,\ldots,N.
\end{equation*}
Recall from Lemma \ref{lemmafingerprintproperty-single} (i) that $\mathcal{X}_i$ is $(q-i)$-uniform. Roughly speaking, the structure of `early' fingerprints (those $\mathcal{X}_i$ for small $i$) are complicated. Nevertheless, Lemma \ref{lemmafin-stru0} gives upper bounds on the total number of members of $\mathcal{F}$ covered by these families. Part (ii) is particularly important for our argument. It shows that if a minimum $t$-cover belongs to a fingerprint at a later round, then it belongs to every preceding fingerprint, which leads to a much stronger estimate. By contrast, the `later' fingerprints consist of smaller sets, allowing us to analyze their structure. Therefore, in Lemmas \ref{lemmafin-stru1-single} and \ref{lemmafin-stru2}, instead of merely estimating the contribution of the terminal fingerprint $\mathcal{X}_N$, we use the structural information encoded by these fingerprints and the minimal $t$-covers to characterize large $t$-intersecting families.
\begin{lemma}\label{lemmafin-stru0}
With the notation in Assumption \ref{assumption}, the following hold. 
	\begin{itemize}
	\item[\rm(i)]$\delta_0\leq\delta_1\leq\cdots\leq\delta_{k-t-2}$, and for all $m\geq\max\{q-N,t+2\}$, it holds that 
	$$
	\left|\bigcup_{i=0}^{q-m}\mathcal{F}[\mathcal{X}_i]\right|\leq r^{m-t-1}\delta_{q-m}\binom{n-m}{k-m}.
	$$
In particular, $\left|\bigcup_{i=0}^{q-m}\mathcal{F}[\mathcal{X}_i]\right|\leq\delta_{q-m}\binom{n-t-1}{k-t-1}$, and $|\mathcal{F}|\leq\delta\binom{n-t-1}{k-t-1}$ for $N\leq q-t-2$.
	\item[\rm(ii)]If $\mathcal{M}^*(\mathcal{F})\cap\mathcal{S}_{q-j}\neq\emptyset$, then for all $m\geq j$, it holds that
	\begin{equation*}
		\left|\bigcup_{i=0}^{q-m}\mathcal{F}[\mathcal{X}_i]\right|\leq\frac{(m-t+1)^{m-t}}{1-\varepsilon}\binom{\tau_t(\mathcal{F})}{t}\binom{n-m}{k-m}.
	\end{equation*}
\end{itemize}
\end{lemma}
\begin{proof}
	(i)\;First, by Lemma \ref{lemmarqt}, for $t+3\leq m\leq k$, we have $\varphi(m-t,t)\leq r^{m-t}T_{m-t}$ and $T_{m-t+1}\leq\varepsilon T_{m-t}$. Then
	\begin{align*}
	\delta_{k-m}&=\varphi(m-t,t)/r^{m-t-1}+rT_{m-t+1}/(1-\varepsilon)\\
	&\leq rT_{m-t}+rT_{m-t}\varepsilon/(1-\varepsilon)=rT_{m-t}/(1-\varepsilon)\leq\delta_{k-m+1}.
	\end{align*}
	Hence $\delta_0\leq\delta_1\leq\cdots\leq\delta_{k-t-2}$.
	
	By Lemmas \ref{lemmafingerprintproperty-single} (iv) and
\ref{lemmarqt},
\begin{align*}
	\left|\bigcup_{i=0}^{q-m}\mathcal{F}[\mathcal{X}_i]\right|\bigg/\binom{n-m}{k-m}
	&\leq\sum_{i=0}^{q-m}|\mathcal{X}_i|\binom{n-(q-i)}{k-(q-i)}\bigg/\binom{n-m}{k-m} \\
	&\leq r^{m-t}\sum_{i=0}^{q-m}\varphi(q-t-i,t)/r^{q-t-i}=r^{m-t}\Phi(r,q,t,m).
\end{align*}
Further, by Lemma \ref{lemmarqt} (i),
\begin{align*}
r^{m-t}\Phi(r,q,t,m)&\leq\varphi(m-t,t)+r^{m-t}\Psi(r,q,t,m+1)\\&\leq\varphi(m-t,t)+r^{m-t}T_{m-t+1}/(1-\varepsilon)=r^{m-t-1}\delta_{q-m},
\end{align*} 
and hence the first upper bound in (i) holds. In particular, from $\binom{n-t-1}{k-t-1}\geq r^{m-t-1}\binom{n-m}{k-m}$, we obtain $\left|\bigcup_{i=0}^{q-m}\mathcal{F}[\mathcal{X}_i]\right|\leq\delta_{q-m}\binom{n-t-1}{k-t-1}$.

Assume $N\leq q-t-2$. By the definition, we get $\mathcal{S}_N=\mathcal{X}_N$. Then Lemma \ref{lemmafingerprintproperty-single} (ii) yields
$$\mathcal{F}\subseteq\left(\bigcup_{i=0}^{N-1}\mathcal{F}[\mathcal{X}_i]\right)\cup\mathcal{F}[\mathcal{S}_N]\subseteq\bigcup_{i=0}^{N}\mathcal{F}[\mathcal{X}_i],$$
and hence $|\mathcal{F}|\leq\delta_{N}\binom{n-t-1}{k-t-1}\leq\delta_{q-t-2}\binom{n-t-1}{k-t-1}=\delta\binom{n-t-1}{k-t-1}$.

(ii)\;Fix  $T\in\mathcal{M}^*(\mathcal{F})\cap\mathcal{S}_{q-j}$. Since $|T|=\tau_t(\mathcal{F})$ and every member of  $\mathcal{S}_{q-j}$ has size at most $j$, we have $j\geq\tau_t(\mathcal{F})$. Next, we find that $$T\in\mathcal{S}_{i},\;i=0,1,\ldots,q-j.$$ 
To see this, note that in each step before  $\mathcal{S}_{q-j}$ is output, the algorithm deletes some sets with size at least $q-(q-j-1)=j+1>\tau_t(\mathcal{F})$, and proceed by taking a fingerprint of the remaining sets. Hence, there exists a sequence of sets $T=T_0\supseteq T_1\supseteq\cdots\supseteq T_{q-j}$ such that $T_i\in\mathcal{S}_i$ for $i\leq q-j$. Since $\mathcal{S}_{q-j}$ is an antichain, we deduce that $T_{q-j}=T$, and hence the $T_i$'s coincide. Now for all $i\leq q-j$, the set $T$ forms a $t$-cover of $\mathcal{X}_i$, and so $\mathcal{X}_i=\bigcup_{Z}\mathcal{X}_i[Z]$, where $Z$ ranges over $\binom{T}{t}$. It follows from Lemma \ref{lemmafingerprintproperty-single} (iii) that
\begin{equation*}
|\mathcal{X}_i|\leq\binom{\tau_t(\mathcal{F})}{t}(q-t-i+1)^{q-t-i}.
\end{equation*} 
Using this improved bound, the same argument as (i) yields
\begin{align*}
	\left|\bigcup_{i=0}^{q-m}\mathcal{F}[\mathcal{X}_i]\right|\bigg/\binom{n-m}{k-m}&\leq r^{m-t}\binom{\tau_t(\mathcal{F})}{t}\sum_{i=0}^{q-m}\left(\frac{q-t-i+1}{r}\right)^{q-t-i}\\
	&=r^{m-t}\binom{\tau_t(\mathcal{F})}{t}\sum_{i=m}^{q}\left(\frac{i-t+1}{r}\right)^{i-t}\leq\frac{(m-t+1)^{m-t}}{1-\varepsilon}\binom{\tau_t(\mathcal{F})}{t}.
\end{align*}
To see the last inequality, write $a_i=\left(\frac{i-t+1}{r}\right)^{i-t}$ for $i=m,m+1,\ldots,q$. Then 
$$\frac{a_{i+1}}{a_i}=\frac{i-t+1}{r}\cdot\left(1+\frac{1}{i-t+1}\right)^{i-t+1}<\frac{eq}{r}\leq\varepsilon$$
as $\varepsilon r\geq eq$, and hence $$r^{m-t}\sum_{i=m}^qa_i\leq r^{m-t} a_m(1+\varepsilon+\varepsilon^{q-m})<(m-t+1)^{m-t}/(1-\varepsilon),$$
as required.
\end{proof}
\begin{lemma}\label{lemmafin-stru1-single}
With the notation in Assumption \ref{assumption}, and further suppose $N=q-t-1$. The following hold.
\begin{itemize}
\item[\rm(i)]$\mathcal{S}_{q-t-2}=\mathcal{S}_{q-t-1}$, $\mathcal{S}^*\subseteq\mathcal{S}_{q-t-1}$, and $\mathcal{S}_{q-t-1}\subseteq\binom{Z}{t+1}$ for some $(t+2)$-subset $Z$.
\item[\rm(ii)]Either $|\mathcal{S}^*|\geq3$ and $\mathcal{F}\cong\mathcal{A}(n,k,t)$, or $k\geq t+3$,  $|\mathcal{S}^*|\leq2$ and $|\mathcal{F}|\leq\max\{M_1,M_2\}\cdot\binom{n-t-1}{k-t-1}$, where $M_1:=\delta_{q-t-4}+\alpha(t+2)+32(t+2)^2/r^2$ and $M_2:=2+t\alpha+\frac{64(t+1)}{(1-\varepsilon)r^2}$.
\end{itemize}
\end{lemma}
\begin{proof}
(i)\;Write
$\mathcal{S}=\mathcal{S}_{q-t-1}$ for short. 	
Since  $\mathcal{S}$ is a $(t+1)$-uniform $t$-intersecting antichain, $\mathcal{S}$ is non-trivial. It follows from the minimality that $\mathcal{S}\subseteq\binom{Z}{t+1}$ for some  $Z\in\binom{[n]}{t+2}$. Since $N=q-t-1$, the algorithm terminates at round $q-t-1$, and hence 	$\mathcal{S}=\mathcal{X}_{q-t-1}$ by the definition of $N$. Since $\mathcal{S}$ is a fingerprint of $\mathcal{S}_{q-t-2}\subseteq\binom{[n]}{\leq(t+1)}$ and $\mathcal{S}$ is $(t+1)$-uniform, we have $\mathcal{S}_{q-t-2}\setminus\mathcal{X}_{q-t-2}=\mathcal{S}$. 

To prove $\mathcal{S}_{q-t-2}=\mathcal{S}$, it suffices to prove $\mathcal{X}_{q-t-2}=\emptyset$. Indeed, assume to the contrary that it is non-empty. Let $W\in\mathcal{X}_{q-t-2}$. Note that $\mathcal{S}$ has at least three members and each of them $t$-intersects $W$. It follows that $|W\cap Z|\geq t+1$, and hence $|W\cap Z|=t+1$ as $\mathcal{S}_{q-t-2}$ is an antichain. So every set from $\mathcal{S}_{q-t-2}$ contains all but at most one element of $Z$, and hence every $(t+1)$-subset of $Z$ is a $t$-cover of $\mathcal{S}_{q-t-2}$. However, this contradicts the minimality of the members of $\mathcal{X}_{q-t-2}$. Hence $\mathcal{X}_{q-t-2}$ must be empty. This yields $S_{q-t-2}=\mathcal{S}$.
	
Note that the sets removed before
round $q-t-1$ have sizes at least $t+2$. Then no member of $\mathcal{S}^*$ is removed, and hence each of them contains some from $\mathcal{S}$,  implying $\mathcal{S}^*\subseteq\mathcal{S}$.
	
(ii)\;If $|\mathcal{S}^*|\geq3$, then $|F\cap Z|\geq t+1$ for all 
	$F\in\mathcal{F}$. To see this, suppose for contradiction
	that $|F\cap Z|=t$ for some $F\in\mathcal{F}$. Since every member of
	$\mathcal{S}^*$ is a $t$-cover of $\mathcal{F}$ and
	$\mathcal{S}^*\subseteq\binom{Z}{t+1}$, each of them 
	must contain the same $t$-set $F\cap Z$. However, there are only two
	$(t+1)$-subsets of $Z$ containing a fixed $t$-subset of $Z$, a
	contradiction. Hence $|F\cap Z|\geq t+1$ for every $F\in\mathcal{F}$,
	and so $\mathcal{F}\lesssim\mathcal{A}(n,k,t)$. By the maximality of
	$\mathcal{F}$, we have $\mathcal{F}\cong\mathcal{A}(n,k,t)$.
	
	It remains to consider the case $|\mathcal{S}^*|\leq2$. Now $k\geq t+3$, as otherwise $\mathcal{S}_{q-t-2}=\mathcal{S}_0$ contains at least three sets of size $t+1$. By Lemma \ref{lemmafingerprintproperty-single} (ii) and the fact that $\mathcal{S}_{q-t-2}=\mathcal{S}$, 
	$$\mathcal{F}\subseteq\left(\bigcup_{i=0}^{q-t-3}\mathcal{F}[\mathcal{X}_i]\right)\cup\mathcal{F}[\mathcal{S}].$$ 
	
	Let us estimate $|\mathcal{F}[\mathcal{S}]|$. Since $\mathcal{S}\subseteq\binom{Z}{t+1}$, of course 
	$|\mathcal{S}|\leq t+2$. For each $S\in\mathcal{S}^*$, we trivially have
	$|\mathcal{F}[S]|\leq\binom{n-t-1}{k-t-1}$. On the other hand, if
	$W\in\mathcal{S}\setminus\mathcal{S}^*$, then $W$ is not a $t$-cover of
	$\mathcal{F}$. Hence Lemma \ref{lemmaind} gives 
\begin{equation}\label{equlemmafin-stru11}
|\mathcal{F}[W]|\leq
b\binom{n-t-2}{k-t-2}\leq\alpha\binom{n-t-1}{k-t-1}\;\;\mbox{for all}\;\;W\in\mathcal{S}\setminus\mathcal{S}^*.
\end{equation}	
 
Note that  $\mathcal{S}_{q-t-3}\setminus\mathcal{X}_{q-t-3}$ is non-empty, because every set from  $\mathcal{S}_{q-t-2}=\mathcal{S}$ is contained in some of its members. Fix an   $S\in\mathcal{S}_{q-t-3}\setminus\mathcal{X}_{q-t-3}$. The set $S$ has size at most $t+2$, and it $t$-intersects every $W\in\mathcal{X}_{q-t-3}$. By Lemma \ref{lemmafingerprintproperty-single} (iii), with $T$ ranging over $\binom{S}{t}$, 
$$|\mathcal{X}_{q-t-3}|\leq\sum_{T}|\mathcal{X}_{q-t-3}[T]|\leq4^3\binom{t+2}{2}.$$
Hence, if $\mathcal{S}^*=\emptyset$, then by Lemma \ref{lemmafin-stru0} (i), (\ref{equlemmafin-stru11}) and the bound above,
\begin{align*}
|\mathcal{F}|&\leq\left|\bigcup_{i=0}^{q-t-4}\mathcal{F}[\mathcal{X}_i]\right|+|\mathcal{F}[\mathcal{X}_{q-t-3}]|+|\mathcal{F}[\mathcal{S}]|\\
&\leq r^3\delta_{q-t-4}\binom{n-t-4}{k-t-4}+64\binom{t+2}{2}\binom{n-t-3}{k-t-3}+\alpha(t+2)\binom{n-t-1}{k-t-1}\\&<(\delta_{q-t-4}+\alpha(t+2)+32(t+2)^2/r^2)\binom{n-t-1}{k-t-1}=M_1\binom{n-t-1}{k-t-1}.
\end{align*}

If $|\mathcal{S}^*|\in\{1,2\}$, then $\tau_t(\mathcal{F})=t+1$ and  $\mathcal{M}^*(\mathcal{F})=\mathcal{S}^*\subseteq\mathcal{S}_{q-t-2}$ from (i). By applying Lemma \ref{lemmafin-stru0} (ii) with $m=t+3$, we obtain
\begin{align}
	|\mathcal{F}|&\leq\left|\bigcup_{i=0}^{q-t-3}\mathcal{F}[\mathcal{X}_i]\right|+|\mathcal{F}[\mathcal{S}^*]|+|\mathcal{F}[\mathcal{S}]\setminus\mathcal{F}[\mathcal{S}^*]|\nonumber\\
	&\leq\frac{64(t+1)}{1-\varepsilon}\binom{n-t-3}{k-t-3}+\left(|\mathcal{S}^*|+(t+2-|\mathcal{S}^*|)\alpha\right)\binom{n-t-1}{k-t-1}\nonumber\\
	&\leq\left(\frac{64(t+1)}{(1-\varepsilon)r^2}+2+t\alpha\right)\binom{n-t-1}{k-t-1}=M_2\binom{n-t-1}{k-t-1},\label{equlemmafin-stru25}
\end{align}
where in the last step we used $|\mathcal{S}^*|\leq2$ and $r\geq b$. 
\end{proof}
\begin{lemma}\label{lemmafin-stru2}
With the notation in Assumption \ref{assumption}, and further suppose $N=q-t$. Then there are $j\in[q-t-1]$ and $X\in\binom{[n]}{t}$ such that $j=\min\{i\in[q-t-1]:\mathcal{S}_{i}=\{X\}\}$, and the following hold.
\begin{itemize}
\item[\rm(i)]If $\tau_t(\mathcal{F})\geq t+2$, then $|\mathcal{F}|\leq(\delta+\alpha b)\binom{n-t-1}{k-t-1}$.
\item[\rm(ii)]Suppose $n\geq t+6\cdot\max\{(t+2)^2, k(k-t)\}$. If $\tau_t(\mathcal{F})=t+1$ and $|\mathcal{F}|\geq\binom{n-t}{k-t}-\binom{n-k}{k-t}$, then $j\in\{1,2\}$. Moreover, either $\mathcal{F}\cong\mathcal{L}(n,k,t,s)$ for some $s\in\{k-t,k-t+1\}$, or $\mathcal{F}\cong\mathcal{K}(n,k,t,c)$ for some $c\in[2,k-t]$.
\end{itemize}
\end{lemma}
\begin{proof}
Since $N=q-t$ and $\mathcal{S}_{q-t}$ is a $t$-intersecting antichain, $\mathcal{S}_{q-t}=\{X\}$ for some $t$-subset $X$. Every member of $\mathcal{S}_{q-t-1}\setminus\mathcal{X}_{q-t-1}$ contains $X$, and hence $\mathcal{S}_{q-t-1}=\{X\}$ as well. So we may define $j$ to be the smallest index satisfying this, that is,
$$j:=\min\{i\in[q-t-1]:\mathcal{S}_{i}=\{X\}\}.$$	
To see $j>0$, just note that $\mathcal{F}$ is non-trivial, which implies that $\mathcal{S}_0$ does not contain any  $t$-subset. 

Since $\mathcal{S}_j=\{X\}$, we have $X\subseteq\cap(\mathcal{S}_{j-1}\setminus\mathcal{X}_{j-1})$. By the definition of $j$, the set $X$ does not lie in $\mathcal{S}_{j-1}$, and then the minimality yields  $\mathcal{X}_{j-1}[\overline{X}]\neq\emptyset$. 

(i)\;First, from $\tau_t(\mathcal{F})\geq t+2$, Lemma \ref{lemmakey} yields $$|\mathcal{F}[X]|\leq(k-t+1)^2\binom{n-t-2}{k-t-2}\leq\alpha b\binom{n-t-1}{k-t-1}.$$
By Lemma \ref{lemmafingerprintproperty-single} (ii) and the fact that $X\subseteq\cap(\mathcal{S}_{j-1}\setminus\mathcal{X}_{j-1})$, we obtain
\begin{equation*}
\mathcal{F}\subseteq\left(\bigcup_{i=0}^{j-1}\mathcal{F}[\mathcal{X}_i]\right)\cup\mathcal{F}[X],
\end{equation*}
and this together with Lemma \ref{lemmafin-stru0} (i)  gives $|\mathcal{F}|\leq(\delta+\alpha b)\binom{n-t-1}{k-t-1}$.

(ii)\;From $n\geq t+6\cdot\max\{(t+2)^2, k(k-t)\}$, we have $r\geq6\cdot\max\{(t+2)^2/(k-t), k\}$.

Since $\tau_t(\mathcal{F})=t+1$, the family $\mathcal{S}^*$ is nothing but $\mathcal{M}^*(\mathcal{F})$, that is, the set of $t$-covers of $\mathcal{F}$ with size $t+1$. Then by the maximality of $\mathcal{F}$, we obtain
\begin{equation}\label{equlemmafin-stru22}
\mathcal{F}[\mathcal{S}^*]=\left\{F\in\binom{[n]}{k}:S\subseteq F\;\mbox{for some}\;S\in\mathcal{S}^*\right\}.
\end{equation}
Note that $\mathcal{S}_{q-t-1}=\{X\}$, the sets removed in the procedure are those from $\cup_{j\leq q-t-2}\mathcal{X}_j$, where each has size at least $t+2$. So no member of $\mathcal{S}^*$ is removed. Note also that none of them is replaced by a proper subset, since they are of size $t+1$ and each $\mathcal{S}_i$ ($0\leq i\leq j-1$) is a $t$-intersecting antichain. Hence $$\mathcal{S}^*\subseteq\mathcal{S}_{j-1}\setminus\mathcal{X}_{j-1}.$$
In particular, this gives $X\subseteq\cap\mathcal{S}^*$. To ease notation, write $\mathcal{S}=\mathcal{S}_{j-1}$ and $b(j):=q-t-j+2$ for short. We cover $\mathcal{F}$ by the following subfamilies.
\begin{equation}\label{equlemmafin-stru2decomp}
\mathcal{F}\subseteq\left(\bigcup_{i=0}^{j-2}\mathcal{F}[\mathcal{X}_i]\right)\cup\mathcal{F}[\mathcal{S}^*]\cup(\mathcal{F}[\mathcal{S}[X]]\setminus\mathcal{F}[\mathcal{S}^*])\cup\mathcal{F}[\mathcal{X}_{j-1}[\overline{X}]].
\end{equation}
Write $\mathcal{F}'=\left(\bigcup_{i=0}^{j-2}\mathcal{F}[\mathcal{X}_i]\right)\cup(\mathcal{F}[\mathcal{X}_{j-1}[\overline{X}]])$ and $\mathcal{F}''=\mathcal{F}[\mathcal{S}[X]]\setminus\mathcal{F}[\mathcal{S}^*]$ for short. 

 By applying Lemma \ref{lemmafin-stru0} (ii) to $m=q-j+1$, and noting that $\tau_t(\mathcal{F})=t+1$, $r\geq eq/\varepsilon\geq2ek$ and $\varepsilon\leq0.5$, we obtain
\begin{align}
	|\mathcal{F}'|&\leq\left|\bigcup_{i=0}^{j-1}\mathcal{F}[\mathcal{X}_i]\right|\leq\frac{(q-j-t+2)^{q-j-t+1}(t+1)}{1-\varepsilon}\binom{n-(q-j+1)}{k-(q-j+1)}\nonumber\\
	&=\frac{b(j)^{b(j)-1}(t+1)}{1-\varepsilon}\binom{n-t-b(j)+1}{k-t-b(j)+1}\leq\frac{(t+1)r^2}{1-\varepsilon}\cdot\left(\frac{b(j)}{r}\right)^{b(j)-1}\binom{n-t-2}{k-t-2}\nonumber\\
	&\leq\frac{64(t+1)}{(1-\varepsilon)r}\binom{n-t-2}{k-t-2}<20\binom{n-t-2}{k-t-2}\;\;\mbox{for}\;\;1\leq j\leq q-t-2,\label{equlemmafin-stru23}
\end{align}
where in the last step we used $r=ek/\varepsilon\geq6k$ to derive 
\begin{equation}\label{equlemmafin-stru27}
\frac{64(t+1)}{(1-\varepsilon)r}\leq64(t+1)/(r-ek)<64/(6-e)<20.
\end{equation}
Suppose $j=q-t-1$, and fix an $S\in\mathcal{S}^*$.  Then each member of $\mathcal{X}_{q-t-2}[\overline{X}]$ contains all but some $x\in X$ of $S$. This together with Lemma \ref{lemmafingerprintproperty-single} (iii) yields
\begin{align*}
|\mathcal{X}_{q-t-2}[\overline{X}]|\leq\sum_{x\in X}|\mathcal{X}_{q-t-2}[S\setminus\{x\}]|\leq t\cdot3^2=9t.
\end{align*}
Therefore, by Lemma \ref{lemmafin-stru0} (ii) and (\ref{equlemmafin-stru27}),
\begin{align}
|\mathcal{F}'|&\leq\left|\bigcup_{i=0}^{q-t-3}\mathcal{F}[\mathcal{X}_i]\right|+|\mathcal{F}[\mathcal{X}_{q-t-2}[\overline{X}]]|\leq\left(\frac{64(t+1)}{(1-\varepsilon)r}+9t\right)\binom{n-t-2}{k-t-2}\nonumber\\
&<(20+9t)\binom{n-t-2}{k-t-2}\;\;\mbox{for}\;\;j=q-t-1.\label{equlemmafin-stru24}
\end{align}

We next consider  $\mathcal{F}''$. Set $P=(\cup\mathcal{S}^*)\setminus X$ and put $s^*=|\mathcal{S}^*|$, then clearly $\mathcal{S}^*=\{X\cup\{p\}:p\in P\}$. By the minimality of $\mathcal{S}$, we have $\mathcal{S}[\overline{X}]\neq\emptyset$. Fix a  $W_0$ from the part. Then $|W_0\cap X|=t-1$ and $P\subseteq W_0$ as it $t$-intersects every member of $\mathcal{S}^*$. It follows that $(t-1)+s^*\leq q-(j-1)$, which gives $s^*\leq q-t-j+2=b(j)$. Every $G\in\mathcal{F}''$ intersects $X\cup P$ exactly in $X$, and then must contain some element of $W_0$. Hence 
$$\mathcal{F}''\subseteq\bigcup_{w\in W_0\setminus(X\cup P)}\mathcal{F}[X\cup\{w\}].$$
Note that $|W_0\setminus(X\cup P)|=(q-j+1)-(t-1)-s^*=b(j)-s^*$. By the same argument, every $X\cup\{w\}$ as displayed above does not lie in $\mathcal{S}^*$, and hence Lemma \ref{lemmaind} gives
$$|\mathcal{F}''|\leq (b(j)-s^*)(k-t+1)\binom{n-t-2}{k-t-2}=(b(j)-s^*)b\binom{n-t-2}{k-t-2}.$$
By combining this with (\ref{equlemmafin-stru22}) and  (\ref{equlemmafin-stru2decomp}), we derive that
\begin{align}
|\mathcal{F}[\mathcal{S}^*]|+|\mathcal{F}''|\leq\left(\binom{n-t}{k-t}-\binom{n-t-s^*}{k-t}\right)+(b(j)-s^*)b\binom{n-t-2}{k-t-2}=:g(j,s^*).\label{equlemmafin-stru28}
\end{align}
By $r\geq2ek$, it is easy to check that the function $g(j,s)$ is increasing as $s\in\{1,2,\ldots,b(j)\}$ increases.

In what follows, suppose $|\mathcal{F}|\geq\binom{n-t}{k-t}-\binom{n-k}{k-t}$. We are ready to prove the following technical claim.
\begin{claim}\label{claimlemmafin-stru22}
	Either $(j,s^*)=(2,k-t)$ or $j=1$ and $s^*\in\{k-t,k-t+1\}$. 
\end{claim}
\begin{proof}
For simplicity, write 
\begin{equation*}
	D:=\frac{\left(\binom{n-t}{k-t}-\binom{n-k}{k-t}\right)-|\mathcal{F}|}{\binom{n-t-2}{k-t-2}}.
\end{equation*} 

Assume that $3\leq j\leq q-t-1$. Now necessarily $k\geq t+4$ and $b(j)\leq k-t-1$, and therefore 
$$g(j,s^*)\leq g(j,b(j))=\binom{n-t}{k-t}-\binom{n-t-b(j)}{k-t}\leq\binom{n-t}{k-t}-\binom{n-k+1}{k-t}.$$
It follows that
\begin{align*}
	|\mathcal{F}|\leq&|\mathcal{F}[\mathcal{S}^*]|+|\mathcal{F}''|+|\mathcal{F}'|\leq g(j,s^*)+|\mathcal{F}'|\nonumber\\
	<&\binom{n-t}{k-t}-\binom{n-k+1}{k-t}+\left(20+9t\right)\binom{n-t-2}{k-t-2},
\end{align*}
and hence $D>\binom{n-k}{k-t-1}-\left(20+9t\right)\binom{n-t-2}{k-t-2}$. By Lemma \ref{lemmabinom} (ii), we have
\begin{align}\label{equlemmafin-stru26}
	\binom{n-k}{k-t-1}\bigg/\binom{n-t-2}{k-t-2}&\geq\left(1-\frac{(k-t-1)^2}{n-t-1}\right)\cdot\frac{n-t-1}{k-t-1}=\frac{n-t-1}{k-t-1}-(k-t-1).
\end{align}
From $n\geq t+6\cdot\max\{(t+2)^2, k(k-t)\}\geq t+1.5(t+2)^2+4.5k(k-t)$,
\begin{align*}
	(k-t-1)D&>((n-t-1)-(k-t-1)^2)-(k-t-1)(20+9t)\\
	&\geq1.5(t+2)^2+4.5k(k-t)-1-(k-t-1)(k+8t+19)\\
	&=\frac72\left((k-t)-\frac{9t+36}{14}\right)^2
	+\frac{3t^2+192t+48}{56}>0.
\end{align*}
Therefore, by our assumption that $|\mathcal{F}|\geq\binom{n-t}{k-t}-\binom{n-k}{k-t}$, we obtain $j\in\{1,2\}$.

To prove the claim, we next assume to the contrary that $j\in\{1,2\}$ and $s^*\leq k-t-1$. When $j\leq q-t-2$, a similar argument using (\ref{equlemmafin-stru23}) and (\ref{equlemmafin-stru28}) yields
$$|\mathcal{F}|\leq g(j,s^*)+|\mathcal{F}'|\leq\binom{n-t}{k-t}-\binom{n-k+1}{k-t}+\left(2b+20\right)\binom{n-t-2}{k-t-2}.$$
From (\ref{equlemmafin-stru26}), $n-t\geq6k(k-t)$, and $k\geq t+3$ (as $1\leq j\leq q-t-2$), we obtain
\begin{align*}
	(k-t-1)D&>((n-t-1)-(k-t-1)^2)-(20+2b)(k-t-1)\\
	&\geq6k(k-t)-1-(k-t-1)(3(k-t)+21)\\
	&=3(k-t-2)^2+6(k-t)(t-1)+8>0,
\end{align*}
a contradiction. Hence $j=q-t-1$, and so either  $(k-t,j)=(2,1)$ or $(k-t,j)=(3,2)$. If $(k-t,j)=(3,2)$, then $b=k-t+1=4$, and by (\ref{equlemmafin-stru24}),
\begin{align*}
	|\mathcal{F}|\leq g(2,2)+|\mathcal{F}'|\leq\binom{n-t}{3}-\binom{n-k+1}{3}+\left(24+9t\right)\binom{n-t-2}{1}.
\end{align*}
Hence
$$D\geq\binom{n-(t+3)}{2}/(n-t-2)-(9t+24)>0,$$
a contradiction again. To see the second step, just note that $n-t\geq6k(k-t)=18t+54$. If $(k-t,j)=(2,1)$, then $b(j)=k-t-j+2=3$, and $1\leq s^*\leq k-t-1$ gives $s^*=1$. Moreover, (\ref{equlemmafin-stru24}) reduces to $|\mathcal{F}'|\leq9t$, and (\ref{equlemmafin-stru28}) gives $|\mathcal{F}[\mathcal{S}^*]|+|\mathcal{F}''|\leq\left(\binom{n-t}{2}-\binom{n-t-1}{2}\right)+2b$. Thus $$D\geq\binom{n-t-1}{2}-\binom{n-(t+2)}{2}-(9t+2b)=n-(10t+8)>0,$$
a contradiction again. So Claim \ref{claimlemmafin-stru22} is true.
\end{proof}

Suppose $j=1$. Now $s^*\in\{k-t,k-t+1\}$, and $$\mathcal{X}_{j-1}=\mathcal{X}_0=\mathcal{M}(\mathcal{F})\cap\mathcal{F}\;\;\mbox{and}\;\;\mathcal{S}_0\setminus\mathcal{X}_0=\mathcal{M}(\mathcal{F})\setminus\mathcal{F}.$$
To ease notation, write $\mathcal{M}=\mathcal{M}(\mathcal{F})$. Since $\mathcal{M}$ is a fingerprint of $\mathcal{F}$, a member of $\mathcal{F}$ contains $X$ if and only if it contains some set from $\mathcal{M}[X]$. Moreover, since every $S\in\mathcal{S}_{j-1}\setminus\mathcal{X}_{j-1}$ contains $X$, it follows that  $$\mathcal{F}[\overline{X}]=\mathcal{X}_0[\overline{X}].$$
Note that for all $F\in\mathcal{X}_0[\overline{X}]$, there exists $x\in X$ such that $F\cap(X\cup P)=(X\cup P)\setminus\{x\}$. 

Assume $s^*=k-t+1$. Then $\mathcal{F}[\overline{X}]\subseteq\{(X\setminus\{x\})\cup P:x\in X\}$, which implies that $\mathcal{F}[X]=\mathcal{F}[\mathcal{S}^*]$. Indeed, fix $W\in\mathcal{F}[\overline{X}]$. For each $F\in\mathcal{F}[X]$, it holds that $|F\cap W\cap X|=t-1$ and hence $F\cap P\neq\emptyset$, namely, $F$ contains some $S\in\mathcal{S}^*$. Thus $\mathcal{F}[X]=\mathcal{F}[\mathcal{S}^*]$. Then we deduce from the maximality of $\mathcal{F}$ that
\begin{equation*}
\mathcal{F}=\left\{F\in\binom{[n]}{k}:S\subseteq F\;\mbox{for some}\;S\in\mathcal{S}^*\right\}\cup\{(X\setminus\{x\})\cup P:x\in X\}.
\end{equation*}
Thus $\mathcal{F}\cong\mathcal{H}(n,k,t)$. 

Assume $s^*=k-t$, then $$\mathcal{F}[\overline{X}]\subseteq\{(X\setminus\{x\})\cup P\cup\{y\}:x\in X,\;y\in[n]\setminus(X\cup P)\}.$$
Denote by $Q$ the set of elements of  $[n]\setminus(X\cup P)$ that are contained in some $F\in\mathcal{F}[\overline{X}]$. Then $Q$ is not a singleton. To see this, pick $w\in Q$. Since $w\notin P$, the set $X\cup\{w\}$ is not a $t$-cover of $\mathcal{F}$, and consequently $|F\cap(X\cup\{w\})|<t$ for some $F_0\in\mathcal{F}[\overline{X}]$. It follows that $F_0$ contains an element in $Q$ that is distinct from $w$, and thus $|Q|\geq2$. Pick a  $W=(X\setminus\{x\})\cup P\cup\{y\}\in\mathcal{F}[\overline{X}]$. By the definition of $\mathcal{M}$, the set $(X\setminus\{x\})\cup P$ is not a $t$-cover of $\mathcal{F}$. Then $\mathcal{F}[X]\nsubseteq\mathcal{F}[\mathcal{S}^*]$. Hence, there exists $G\in\mathcal{F}[X]$ with $G\cap P=\emptyset$, implying that $Q\subseteq G$. Thus $|Q|\leq k-t$. Again, by the maximality, we get $\mathcal{F}\cong\mathcal{K}(n,k,t,|Q|)$.

Finally, let us suppose $(j,s^*)=(2,k-t)$. Now
$$\mathcal{X}_1[\overline{X}]\subseteq\{(X\setminus\{x\})\cup P:x\in X\}.$$
Recall that $\mathcal{S}^*=\{X\cup\{p\}:p\in P\}$, and both $\mathcal{S}^*$ and $\mathcal{X}_1[\overline{X}]$ are subfamilies of $\mathcal{S}_0=\mathcal{M}$. Further, every  $W\in\mathcal{M}[X]$ must intersect $P$, and hence contains some member of $\mathcal{S}^*$. Then $\mathcal{M}[X]=\mathcal{S}^*$ as $\mathcal{M}$ is an antichain. Let $W\in\mathcal{M}[\overline{X}]$, then $|W\cap X|=t-1$ and $P\subseteq W$. Therefore, from Lemma \ref{lemmamaximal},
\begin{align*}
\mathcal{F}\subseteq&\left\{F\in\binom{[n]}{k}:S\subseteq F\;\mbox{for some}\;S\in\mathcal{S}^*\right\}\\
&\cup\{(X\setminus\{x\})\cup P\cup\{y\}:x\in X,\;y\in[n]\setminus(X\cup P)\}\cong\mathcal{L}(n,k,t,k-t).
\end{align*}
Hence $\mathcal{F}\cong\mathcal{L}(n,k,t,k-t)$ as it is maximal.
\end{proof}
\noindent{\bf Proof of Theorem \ref{thmlarge}}.\;Our goal is to prove that, if $\mathcal{F}$ is isomorphic to none of the families given in (i) and (ii), then it has size smaller than $\binom{n-t}{k-t}-\binom{n-k}{k-t}$.

By our assumption on $n$, we have
\begin{equation}\label{equthmlarge1}
r:=(n-t)/(k-t)\geq6\cdot\max\{(t+2)^2/(k-t), k\}.
\end{equation}
We may adapt the notation in Assumption \ref{assumption}. More precisely, perform the algorithm to
$\mathcal{F}$, the uniformity $q:=k$ and the fingerprint 
$\mathcal{S}_0:=\mathcal{M}(\mathcal{F})$, and let $N$ be the number of
rounds and $(\mathcal{S}_i,\mathcal{X}_i)$, $i\leq N$, be the output
families. 

For simplicity, write $C:=(k-t)-\frac{k-t-1}{n-t-1}\binom{k-t}{2}$.
A useful estimate obtained by applying Lemma \ref{lemmabinom} (iii) to  $m=n-t$ and $u=v=k-t$ is that
\begin{align}
	\binom{n-t}{k-t}-\binom{n-k}{k-t}&\geq\left((k-t)-\frac{k-t-1}{n-t-1}\binom{k-t}{2}\right)\binom{n-t-1}{k-t-1}=C\binom{n-t-1}{k-t-1}.\label{equbinom}
\end{align}

If $N\leq q-t-2$, then Lemma \ref{lemmafin-stru0} (i) gives $|\mathcal{F}|\leq\delta\binom{n-t-1}{k-t-1}$. When $N=q-t-1$, Lemma \ref{lemmafin-stru1-single} (ii) gives that either 
$\mathcal{F}\cong\mathcal{A}(n,k,t)$, or $k\geq t+3$ and  $|\mathcal{F}|\leq\max\{M_1,M_2\}\cdot\binom{n-t-1}{k-t-1}$, where $M_1:=\delta_{q-t-4}+\alpha(t+2)+32(t+2)^2/r^2$ and $M_2:=2+t\alpha+\frac{64(t+1)}{(1-\varepsilon)r^2}$.

Suppose $N=q-t$. By Lemma \ref{lemmafin-stru2}, if $\tau_t(\mathcal{F})\geq t+2$, then $|\mathcal{F}|\leq(\delta+\alpha b)\binom{n-t-1}{k-t-1}$. If $\tau_t(\mathcal{F})=t+1$ and $|\mathcal{F}|\geq\binom{n-t}{k-t}-\binom{n-k}{k-t}$, then either $\mathcal{F}\cong\mathcal{L}(n,k,t,s)$ for some $s\in\{k-t,k-t+1\}$, or $\mathcal{F}\cong\mathcal{K}(n,k,t,c)$ for some $c\in[2,k-t]$.

Now we combine all obtained inequalities required.  It suffices to prove
\begin{itemize}
\item[\rm(a)]$C>\delta+\alpha b$ for $k\geq t+2$, and
\item[\rm(b)]$C>\max\{M_1,M_2\}$ for $k\geq t+3$.
\end{itemize}
Recall our assumption that $k\geq t+4$. We postpone the tedious computation to Lemma \ref{lemmathmlarge}. 
 {\hfill$\square$}\vspace{1em}

\noindent{\bf Proof of Theorem \ref{thmdiv}.}\;To begin with, we need the following monotonicity property of the parameter $\gamma_t$, that is, $$\gamma_t(\mathcal{H})\leq\gamma_t(\mathcal{H}')\;\;\mbox{for}\;\;\mathcal{H}\subseteq\mathcal{H}'\subseteq\binom{[n]}{k}.$$
Therefore, we may suppose that $\mathcal{F}$ is maximal. We may also suppose that $|\mathcal{F}|\geq\binom{n-t}{k-t}-\binom{n-k}{k-t}$, as the family $\mathcal{L}(n,k,t,k-t)$ has size larger than $\binom{n-t}{k-t}-\binom{n-k}{k-t}$ and $t$-diversity $t(n-k)$. Then by Theorem \ref{thmlarge}, either $\mathcal{F}\cong\mathcal{L}(n,k,t,s)$ for some $s\in\{2,k-t,k-t+1\}$, or $\mathcal{F}\cong\mathcal{K}(n,k,t,c)$ for some $c\in\{2,3,\ldots,k-t\}$. Note that $\mathcal{A}(n,k,t)=\mathcal{L}(n,k,t,2)$ and $\mathcal{H}(n,k,t)=\mathcal{L}(n,k,t,k-t+1)=\mathcal{K}(n,k,t,1)$. Next, we may suppose $\mathcal{F}\ncong\mathcal{A}(n,k,t)$. By Lemma \ref{lemmacountingfamily}, it holds that
\begin{align*}
	|\mathcal{K}(n,k,t,1)|&>|\mathcal{K}(n,k,t,2)|>\cdots>|\mathcal{K}(n,k,t,k-t-2)|\\
	&>|\mathcal{L}(n,k,t,k-t)|>|\mathcal{K}(n,k,t,k-t-1)|>|\mathcal{K}(n,k,t,k-t)|.
\end{align*}
Then the theorem follows from  $\gamma_t(\mathcal{K}(n,k,t,j))=jt$ and $\gamma_t(\mathcal{L}(n,k,t,k-t))=t(n-k)$. {\hfill$\square$}
\begin{lemma}\label{lemmathmlarge}
Let $k\geq t+4$ and $n\geq t+6\max\{(t+2)^2,k(k-t)\}$. Put $C=(k-t)-\frac{k-t-1}{n-t-1}\binom{k-t}{2}$, and let $b,\alpha,\delta,\delta_{q-t-4}$ be as defined in Assumption \ref{assumption}. Then $C>\max\{\delta+\alpha b,M_1,M_2\}$, where $M_1:=\delta_{q-t-4}+\alpha(t+2)+32(t+2)^2/r^2$ and $M_2=2+t\alpha+\frac{64(t+1)}{(1-\varepsilon)r^2}$.
\end{lemma}
\begin{proof}
Recall from Assumption \ref{assumption} that $r=(n-t)/(k-t)\geq6\cdot\max\{(t+2)^2/(k-t), k\}$, and that $b=k-t+1$, $\alpha=b/r$ and $\varepsilon=ek/r\leq0.5$. Also note that $C\geq(k-t)-\binom{k-t}{2}/r>0.9(k-t)$. Let us write $s=t+2$ and $d=k-t$ for short. So our assumption on $n$ yields $r\geq6\cdot\max\{s^2/d,k\}$, and hence
\begin{equation}\label{equrd}
rd\geq\max\{6s^2,6kd\}=\max\{6s^2,6(s+d-2)d\}.
\end{equation}

First, we verify $C\geq\delta+\alpha b$. We have 
\begin{align*}
	\delta r&=\varphi(2,t)+r^2T_3/(1-\varepsilon)=\left(\binom{t}{2}+12t+9\right)+\frac{4^3\binom{t+3}{3}}{(1-\varepsilon)r}\\
	&\leq0.5(t^2+23t+18)+\frac{64(t+2)^3}{3r}=0.5(s^2+19s-24)+\frac{64(t+2)^3}{3r}\\
	&<0.5(s^2+19s-24)+\frac{32}{9}sd<0.5(s^2+19s-24)+3.6sd.
\end{align*}
We will prove that $Cr>(\delta+\alpha b)\cdot r$, or further, using $Cr>0.9(n-t)=0.9rd$, that
\begin{equation}\label{equthmlarge2}
D:=0.9rd-(0.5(s^2+19s-24)+3.6sd+(d+1)^2)>0.
\end{equation}
By (\ref{equrd}), we derive that
\begin{align*}
0.9rd\geq5.4\cdot\max\{s^2,(s+d-2)d\}\geq0.9s^2+4.5(s+d-2)d
\end{align*}
Then from $d\geq4$, we get
\begin{align*}
D&\geq(0.9s^2+4.5(s+d-2)d)-(0.5(s^2+19s-24)+3.6sd+(d+1)^2)\\
&=3.5d^2+0.9sd+0.4s^2-11d-9.5s+11\\
&=(d-4)(0.9s+3.5d+3)+0.1((2s-15)^2+s+5)>0.
\end{align*}
So (\ref{equthmlarge2}) is true, and thus $C>\delta+\alpha b$.

It remains to verify $C\geq\max\{M_1,M_2\}$. By Lemma \ref{lemmarqt}, we obtain
\begin{align*}
	\delta_{q-t-4}&=\varphi(4,t)/r^3+rT_5/(1-\varepsilon)\leq rT_4+rT_4\cdot\varepsilon/(1-\varepsilon)\leq2rT_4=2\cdot5^4\binom{t+4}{4}/r^3.
\end{align*}
Hence 
$r^3d=(rd)r^2\geq(6s^2)(6k)^2=216s^2k^2$. On the other hand, $\binom{t+4}{4}=\binom{s+2}{4}<s^2k^2/24$ as $k\geq s+2$. It follows that $\delta_{q-t-4}/d<1250/(24\cdot216)<0.25$. Hence
\begin{align*}
	M_1/(k-t)=\delta_{q-t-4}/d+\frac{(d+1)s}{rd}+\frac{32s^2}{r^2d}<0.25+\frac{2s}{r}+\frac{32}{36k}<0.25+\frac{1}{3}+\frac{8}{45}<0.9.
\end{align*}
Thus $C/(k-t)>0.9>M_1/(k-t)$, that is, $C>M_1$. 

For $M_2$, we have
\begin{align*}
	M_2=2+t\alpha+\frac{64(t+1)}{(1-\varepsilon)r^2}\leq2+(s-2)(d+1)/r+128(s-1)/r^2.
\end{align*}
Since $d\geq4$, $k\geq t+4$ and $r\geq6k$, it follows that
\begin{align*}
M_2/d&\leq0.5+1.25(s-2)/r+32(s-1)/r^2<0.5+5/24+8/(9k)<0.9.
\end{align*}
Thus $C>0.9d>M_2$. This finishes the proof. 
\end{proof}
\section{Spread approximation for a large family}\label{secspreadapproximation}
In this section, we prove the removal theorem (Theorem \ref{thmremoval}). The direct peeling procedure argument developed in
Section \ref{secfin-stru} applies when $n$ is at least quadratic in $k$. To
handle the complementary range, we will treat a family of sets of size $O(t\log k)$, which is much smaller than $k$ in that case, and then apply the
peeling procedure to $\mathcal S$. This reduction is based on the spread approximation method. More precisely, Theorem \ref{thmstrongsp} states that, given a family $\mathcal{F}$, we iteratively select small sets with relatively dense links so that, apart from a small error term, every member of $\mathcal{F}$ contains some $S\in\mathcal{S}$. Then we use the spread lemma (Theorem \ref{spreadlemma}) to show that if $\mathcal{F}$ is $t$-intersecting, then so is $\mathcal{S}$. The \emph{Spread Lemma} is a variant, due to Tao \cite{Tao}, of a key ingredient in the breakthrough work of Alweiss, Lovett, Wu and Zhang \cite{sunflower} on the Erd\H{o}s--Rado sunflower conjecture. It was subsequently sharpened by Stoeckl \cite{Stoeckl}. 

Let us recall that two families $\mathcal{A}$ and $\mathcal{B}$ are said to be \emph{cross-intersecting} if $A\cap B\neq\emptyset$ for all $A\in\mathcal{A}$ and $B\in\mathcal{B}$. For a real $p\in[0,1]$, a random subset $W$ of $[n]$ is called  \emph{$p$-random} if each element is included in $W$ independently with probability $p$. 
\begin{theorem}[\cite{Tao,Stoeckl}]\label{spreadlemma}
If $\mathcal{F}\subseteq\binom{[n]}{\leq k}$ is $r$-spread and $W$ is an $(m\delta)$-random subset of $[n]$, then 
$${\rm Pr}[F\subseteq W\;\mbox{for some}\; F\in\mathcal{F}]\geq1-\left(\frac{2}{\log_2(r\delta)}\right)^mk.$$
\end{theorem}
\begin{theorem}\label{thmstrongsp}
	For any $c,d\geq1$ and $M>0$, there exist constants
	$L=L(c,d,M)$ and $C=C(c,d,M)$ such that the following holds. Let $k\geq t+2$ and $Ctk\log_2k\leq n<t+Mk^{d}(k-t)$. If $
	\mathcal F\subseteq\binom{[n]}k$ is $t$-intersecting, then there exists a $t$-intersecting family $\mathcal S\subseteq \binom{[n]}{\leq q}$, where $q=t+\lceil Lt\log_2 k\rceil$, such that $|\mathcal F\setminus \mathcal F[\mathcal S]|\leq k^{-ct}\left(\binom{n-t}{k-t}-\binom{n-k}{k-t}\right).$
\end{theorem}
\begin{proof}
Let $L=L(c,d,M)$ be determined later, and set
$r=2^6(q+\log_2(2k))$.
\begin{claim}\label{claimthmstrongsp1}
Any two  $(r/2)$-spread families in $\binom{[n]}{\leq k}$ are not cross-intersecting.
\end{claim}
\begin{proof}
Suppose that $\mathcal{G}_1,\mathcal{G}_2\subseteq\binom{[n]}{\leq k}$ are $(r/2)$-spread. Pick a random $2$-partition $[n]=W_1\cup W_2$ so that every point is included in a block uniformly, and the choices are independent. Now every $W_i$ has the same distribution as a $(1/2)$-random subset of $[n]$. We apply Theorem \ref{spreadlemma} to $m=\lceil\log_2(2k)\rceil$ and $\delta=(2m)^{-1}$. Note that $0.5r\delta>2^6\log_2(2k)/(2^2\log_2(2k))=2^4$. Then
$$\mathbb{P}[G\subseteq W_i\;\mbox{for some}\;G\in\mathcal{G}_i]\geq1-\left(\frac{2}{\log_2(0.5r\delta)}\right)^mk>1-\left(\frac{1}{2k}\right)\cdot k=\frac{1}{2}$$
for $i=1,2$. Therefore there is a specified $2$-partition $W_1, W_2$ together with $G_1\in\mathcal{G}_1,G_2\in\mathcal{G}_2$ such that $G_i\subseteq W_i$ for $i=1,2$, and certainly $G_1\cap G_2=\emptyset$.
\end{proof}
 We iteratively construct the family $\mathcal{S}$ as follows. Set $\mathcal{F}_1=\mathcal{F}$. For each $i\geq1$, do the following:	
\begin{itemize}
\item[\rm(a)]If $|\mathcal{F}_i|\leq k^{-ct}\left(\binom{n-t}{k-t}-\binom{n-k}{k-t}\right)$, then terminate.
\item[\rm(b)]Otherwise, since $\mathcal{F}_i$ is $t$-intersecting, it is not $r$-spread from Claim \ref{claimthmstrongsp1}, then one can find an inclusion-maximal subset $S_i$ such that $|\mathcal{F}_i[S_i]|>r^{-|S_i|}|\mathcal{F}_i|$.
\item[\rm(c)]If $|S_i|>q$, then terminate.
\item[\rm(d)]Otherwise, let $\mathcal{F}_{i+1}:=\mathcal{F}_i\setminus\mathcal{F}_i[S_i]$.
\end{itemize}
Suppose that the procedure terminates at the $N$-th round, and set $\mathcal{S}=\{S_1,S_2,\ldots,S_{N-1}\}$ to be the resulting family. Note that
$$\mathcal{F}_N=\mathcal{F}\setminus\mathcal{F}[\mathcal{S}].$$
Let us collect two properties of $\mathcal{S}$. 
\begin{claim}\label{claimthmstrongsp2}
The following holds.
\begin{itemize}
\item[\rm(i)]For suitably chosen $L$ and $C$, it holds that  $|\mathcal{F}_{N}|\leq k^{-ct}\left(\binom{n-t}{k-t}-\binom{n-k}{k-t}\right)$.
\item[\rm(ii)]Each $\mathcal{F}_i(S_i)$ is $r$-spread.
\end{itemize}
\end{claim}
\begin{proof}
The proof of (ii) is easy. Indeed, by our procedure, $S_i$ is inclusion-maximal with $|\mathcal{F}_i[S_i]|>r^{-|S_i|}|\mathcal{F}_i|$. Then for each non-empty $T\subseteq[n]\setminus S_i$, we have
\begin{align*}
|(\mathcal{F}_i(S_i))[T]|=|\mathcal{F}_i(S_i\cup T)|\leq r^{-|S_i\cup T|}|\mathcal{F}_i|<r^{-|T|}|\mathcal{F}_i[S_i]|.
\end{align*}
So $\mathcal{F}_i(S_i)$ is $r$-spread.

Let us prove (i). By the definition of $N$, either $|\mathcal{F}_{N}|\leq k^{-ct}\left(\binom{n-t}{k-t}-\binom{n-k}{k-t}\right)$ or $|S_{N}|>q$. For the former case we have $$|\mathcal{F}\setminus\mathcal{F}[\mathcal{S}]|=|\mathcal{F}_{N}|\leq k^{-ct}\left(\binom{n-t}{k-t}-\binom{n-k}{k-t}\right).$$

Suppose $|S_N|>q$. Then we get
\begin{align*}
	|\mathcal F_N|&<r^{|S_N|}|\mathcal{F}_N[S_N]|\leq r^{|S_N|}\binom{n-|S_N|}{k-|S_N|}.
\end{align*}
Let us choose the constants and prove $|\mathcal F_N|\leq  k^{-ct}\left(\binom{n-t}{k-t}-\binom{n-k}{k-t}\right)$. First, we claim that for all sufficiently large $L$,
\begin{equation}\label{equlemmastrongsp1}
	r^t k^{-Lt}\leq\frac{1}{(M+1)k^{ct+d}}
	\quad\mbox{for all }k\mbox{ and }t.
\end{equation} 
Indeed, since $k\geq t+2\geq3$, we have
\begin{align*}
r
&=64\bigl(t+\lceil Lt\log_2 k\rceil+\log_2(2k)\bigr)\leq64(L+4)t\log_2 k\leq64(L+4)k^2.
\end{align*}
Choose $L>c+d+2$ sufficiently large that
$64(L+4)3^{c+d+2-L}<(M+1)^{-1}$.
Such an $L$ exists since the left-hand side tends to zero as
$L\to\infty$. Then
\begin{align*}
r^t k^{-Lt}\cdot k^{ct+d}&\leq\bigl(64(L+4)k^2\bigr)^t k^{ct+d-Lt}=\bigl(64(L+4)k^{c+d/t+2-L}\bigr)^t\\
&\leq\bigl(64(L+4)3^{c+d+2-L}\bigr)^t<(M+1)^{-1},
\end{align*}
which proves (\ref{equlemmastrongsp1}).

After fixing $L$, choose $C\geq128(L+4)$. Then
\begin{equation}\label{equlemmastrongsp2}
	r_0:=\frac{n-t}{k-t}>\frac nk
	\geq Ct\log_2 k\geq2r.
\end{equation}
Therefore, from $|S_N|>q=t+\lceil Lt\log_2 k\rceil$ and
(\ref{equlemmastrongsp2}), we derive that
\begin{align*}
|\mathcal{F}_N|\bigg/\binom{n-t}{k-t}&<r^{|S_N|}r_0^{-(|S_N|-t)}=r^t(r_0/r)^{-(|S_N|-t)}\\
	&\leq r^t2^{-(|S_N|-t)}\leq r^t2^{-Lt\log_2 k}=r^t k^{-Lt}.
	\end{align*}
On the other hand, by applying Lemma \ref{lemmabinom} (i) to $m=n-t$ and $u=v=k-t$, we obtain that
\begin{align*}
	\binom{n-k}{k-t}\bigg/\binom{n-t}{k-t}&\leq1-\frac{(k-t)^2}{(n-t)+(k-t)^2}=1-\frac{k-t}{r_0+k-t}\leq1-\frac{1}{r_0+1}.
\end{align*}
This together with $r_0<Mk^d$ gives
\begin{equation}\label{equlemmastrongsp4}
	\binom{n-t}{k-t}-\binom{n-k}{k-t}\geq\frac{1}{r_0+1}\binom{n-t}{k-t}\geq\frac{1}{(M+1)k^d}\binom{n-t}{k-t}.
\end{equation}
Combining these with (\ref{equlemmastrongsp1}) yields
\begin{align*}
	|\mathcal{F}\setminus\mathcal{F}[\mathcal{S}]|
	&=|\mathcal{F}_N|
	\leq r^t k^{-Lt}\binom{n-t}{k-t}\\
	&\leq(M+1)r^t k^{d-Lt}
	\left(\binom{n-t}{k-t}-\binom{n-k}{k-t}\right)\\
	&\leq k^{-ct}
	\left(\binom{n-t}{k-t}-\binom{n-k}{k-t}\right),
\end{align*}
which finishes the proof of Claim \ref{claimthmstrongsp2}.
\end{proof}
It remains to show that $\mathcal{S}$ is $t$-intersecting.  To the contrary, suppose $|S_u\cap S_v|<t$. We note that here $u$ and $v$ need not be distinct. Write $\mathcal{G}_u=\mathcal{F}_u(S_u)$ and $\mathcal{G}_v=\mathcal{F}_v(S_v)$ for short. Define $\mathcal{G}_u'=\{G\in\mathcal{G}_u: G\cap(S_v\setminus S_u)=\emptyset\}$ and $\mathcal{G}_v'=\{G\in\mathcal{G}_v: G\cap(S_u\setminus S_v)=\emptyset\}$. Since $|S_v|\leq q$ and $r\geq2q$, the $r$-spreadness of $\mathcal{G}_u$ gives
	$$
	|\mathcal{G}_u\setminus\mathcal{G}_u'|\leq
	\sum_{x\in S_v\setminus S_u}|\mathcal{G}_u[\{x\}]|
	\leq|S_v|\cdot(r^{-1}|\mathcal{G}_u|)\leq\frac12|\mathcal{G}_u|.$$
	Hence $
	|\mathcal{G}_u'|\geq\frac12|\mathcal{G}_u|$. Then $\mathcal{G}_u'$ is $(r/2)$-spread as Claim \ref{claimthmstrongsp2} (ii) gives that $\mathcal G_u$ is $r$-spread. More precisely, for all non-empty  $T\subseteq[n]\setminus S_u$, we have
	$$|\mathcal{G}_u'[T]|\leq|\mathcal{G}_u[T]|\leq r^{-|T|}\cdot|\mathcal{G}_u|\leq (2r^{-|T|})\cdot|\mathcal{G}_u'|\leq(r/2)^{-|T|}|\mathcal{G}_u'|.$$
	Hence $\mathcal{G}_u'$ is $(r/2)$-spread. Similarly, $
	|\mathcal G_v'|\geq\frac12|\mathcal G_v|$ and $\mathcal{G}_v'$ is $(r/2)$-spread. By Claim \ref{claimthmstrongsp1}, the families $\mathcal{G}_u'$ and $\mathcal{G}_v'$ are not cross-intersecting, that is, there are $G_u\in\mathcal{G}_u'$ and $G_v\in\mathcal{G}_v'$ such that $G_u\cap G_v=\emptyset$. However, this leads to $G_u\cup S_u, G_v\cup S_v\in\mathcal{F}$ with 
	$$|(G_u\cup S_u)\cap(G_v\cup S_v)|=|S_u\cap S_v|<t,$$
	which contradicts that $\mathcal{F}$ is $t$-intersecting. Thus  $\mathcal S$ must be $t$-intersecting.
\end{proof}
For pairwise disjoint subsets $X, P, I\subseteq[n]$ with $|X|=t$, $|P|=k-t$ and $|I|=2$, define
\begin{align*}
\mathcal{K}(X,P,I)=&\left\{F\in\binom{[n]}{k}:X\subseteq F,\;F\cap P\neq\emptyset\right\}\\
&\cup\left\{F\in\binom{[n]}{k}:X\cup I\subseteq F,\;F\cap P=\emptyset\right\}\\
&\cup\{(X\setminus\{x\})\cup P\cup\{y\}:x\in X,\;y\in I\}.
\end{align*}
The family defined above is isomorphic to $\mathcal{K}(n,k,t,2)$.\vspace{1em}

\noindent{\bf Proof of Theorem \ref{thmremoval}.}\;We want to find a sufficiently large $C=C(\eta,\theta)$ such that,
for $k\geq\max\{t+2,(1+\eta)t\}$ and $n\geq Ctk\log_2 k$,
if $\mathcal{F}\subseteq\binom{[n]}{k}$ is a non-trivial
$t$-intersecting family with
$|\mathcal{F}|\geq(1-\theta)|\mathcal{K}(n,k,t,2)|$,
then either $\mathcal{F}$ is isomorphic to a subfamily of
$\mathcal{A}(n,k,t)$ or $\mathcal{H}(n,k,t)$, or there is a copy
$\mathcal{K}$ of $\mathcal{K}(n,k,t,2)$
(that is, $\mathcal{K}\cong\mathcal{K}(n,k,t,2)$) with
$|\mathcal{F}\setminus\mathcal{K}|
\leq2\theta|\mathcal{K}(n,k,t,2)|$. The proof proceeds as follows. First, for $n=\Omega(k^2)$, we perform the peeling procedure to $\mathcal{F}$ itself, and then prove using structure lemmas (Lemmas \ref{lemmafin-stru1-single} and \ref{lemmafin-stru2}). In the complementary range, we apply Theorem \ref{thmstrongsp} to the family to get a spread approximation $\mathcal{S}$, and then apply the peeling procedure to $\mathcal{S}$.  

To begin with, we set $r_0=(n-t)/(k-t)$. We may suppose that $\mathcal{F}$ is maximal, and it is isomorphic to neither $\mathcal{A}(n,k,t)$ nor $\mathcal{H}(n,k,t)$. Recall that 
$$|\mathcal{K}(n,k,t,2)|=\binom{n-t}{k-t}-\binom{n-k}{k-t}+\binom{n-k-2}{k-t-2}+2t.$$
Also note that Lemma \ref{lemmabinom} (i) gives  $\binom{n-t}{k-t}-\binom{n-k}{k-t}>\frac{r_0(k-t)}{r_0+(k-t)}\binom{n-t-1}{k-t-1}$. 
Hence
\begin{equation}\label{equthmremoval1}
|\mathcal{F}|\geq(1-\theta)|\mathcal{K}(n,k,t,2)|\;\;\mbox{implies}\;\;|\mathcal{F}|>(1-\theta)\cdot\frac{r_0(k-t)}{r_0+(k-t)}\binom{n-t-1}{k-t-1}.
\end{equation}
 
\noindent{\bf Case 1.\;}$n\geq C_1k^2$.

First, let $C_1$ be sufficiently large so that $$r_0:=(n-t)/(k-t)\geq6\max\{(t+2)^2/(k-t),k\}.$$ 
Then we may adapt the notation in Assumption \ref{assumption}, where $r_0$ plays the role of $r$ and $\varepsilon_{\ref{assumption}}:=ek/r_0\leq0.5$ plays the role of $\varepsilon$. More precisely, perform the algorithm to $\mathcal{F}$, the uniformity $q=k$ and the initial fingerprint $\mathcal{S}_0=\mathcal{M}(F)$, and let $N$ be the number of
rounds and $(\mathcal{S}_i,\mathcal{X}_i)$ ($i\leq N$) be the output families. By Lemmas \ref{lemmafin-stru0} (i), \ref{lemmafin-stru1-single} and \ref{lemmafin-stru2}, and by our assumption that  $\mathcal{F}\not\lesssim\mathcal{A}(n,k,t)$ and  $\mathcal{F}\not\lesssim\mathcal{H}(n,k,t)$, one of the following holds.\vspace{0.5em}
\begin{itemize}
\item[\rm(a)]$|\mathcal{F}|\leq\max\{\delta+\alpha b,M_1\}\cdot\binom{n-t-1}{k-t-1}$, where $M_1:=\delta_{q-t-4}+\alpha(t+2)+32(t+2)^2/r_0^2$.
\item[\rm(b)]$N=q-t-1$, $|\mathcal{S}^*|\in\{1,2\}$ and $|\mathcal{F}|\leq M_2\binom{n-t-1}{k-t-1}$, where  $M_2:=\frac{64(t+1)}{(1-\varepsilon_{\ref{assumption}})r_0^2}+2+\alpha t$. 
\item[\rm(c)]$N=q-t$ and $\tau_t(\mathcal{F})=t+1$.\vspace{1em}
\end{itemize}

For (a), by a routine computation, it holds that $\max\{\delta+\alpha b,M_1\}=O_{\eta}(k^2)/r_0$. On the other hand,  $\frac{r_0(k-t)}{r_0+(k-t)}>(1-C_1^{-1})\eta k/(1+\eta)$ for $n\geq C_1k^2$. Hence $$(1-\theta)\cdot\frac{r_0(k-t)}{r_0+(k-t)}>\max\{\delta+\alpha b,M_1\}$$
 for sufficiently large $C_1$ depending on $\eta$ and $\theta$, which implies  $|\mathcal{F}|<(1-\theta)|\mathcal{K}(n,k,t,2)|$.
 
 Suppose (b) holds. Here we need more detail from the proof of Lemma \ref{lemmafin-stru1-single} (ii). We see that the quantity $2\binom{n-t-1}{k-t-1}$ bounds the size of $\mathcal{F}[\mathcal{S}^*]$. More precisely, by the same argument as deriving (\ref{equlemmafin-stru25}), we have 
 \begin{align*}
 	|\mathcal{F}\setminus\mathcal{F}[\mathcal{S}^*]|&\leq\left|\bigcup_{i=0}^{q-t-2}\mathcal{F}[\mathcal{X}_i]\right|+|\mathcal{F}[\mathcal{S}]\setminus\mathcal{F}[\mathcal{S}^*]|\nonumber\\
 	&\leq\frac{64(t+1)}{1-\varepsilon_{\ref{assumption}}}\binom{n-t-2}{k-t-2}+(t+2-|\mathcal{S}^*|)\alpha\binom{n-t-1}{k-t-1}\nonumber\\
 	&\leq\left(\frac{64(t+1)}{(1-\varepsilon_{\ref{assumption}})r_0^2}+\frac{(t+1)b}{r_0}\right)\binom{n-t-1}{k-t-1}=O_{\eta}(k^2)/r_0\binom{n-t-1}{k-t-1}.
 \end{align*}
 We construct a copy $\mathcal{K}(X,P,I)$ of $\mathcal{K}(n,k,t,2)$ as follows. Choose a $t$-subset $X\subseteq\cap\mathcal{S}^*$. Next, choose a $(k-t)$-set $P$ with $X\cap P=\emptyset$ and $(\cup\mathcal{S}^*)\setminus X\subseteq P$,  and pick a $2$-subset $I\subseteq[n]\setminus(X\cup P)$. Then $\mathcal{F}\setminus\mathcal{K}(X,P,I)\subseteq\mathcal{F}\setminus\mathcal{F}[\mathcal{S}^*]$, and hence $$|\mathcal{F}\setminus\mathcal{K}(X,P,I)|\leq|\mathcal{F}\setminus\mathcal{F}[\mathcal{S}^*]|=O_{\eta}(k^2)/r_0\binom{n-t-1}{k-t-1}.$$
 So $|\mathcal{F}\setminus\mathcal{K}(X,P,I)|\leq2\theta|\mathcal{K}(n,k,t,2)|$ for large $C_1$. 

Suppose (c) holds. Now $\mathcal{S}^*\neq\emptyset$ consists of the $t$-covers of $\mathcal{F}$ with size $t+1$. Then by the maximality of $\mathcal{F}$, we obtain
\begin{equation}\label{equlemmafin-stru22copy}
\mathcal{F}[\mathcal{S}^*]=\left\{F\in\binom{[n]}{k}:S\subseteq F\;\mbox{for some}\;S\in\mathcal{S}^*\right\}.
\end{equation}
By Lemma \ref{lemmafin-stru2}, there are $j\in[q-t-1]$ and $X\in\binom{[n]}{t}$ such that $j=\min\{i\in[q-t-1]:\mathcal{S}_{i}=\{X\}\}$. Since $\mathcal{S}_{q-t-1}=\{X\}$, the sets removed in the procedure are those from $\cup_{j\leq q-t-2}\mathcal{X}_j$, where each has size at least $t+2$. So no member of $\mathcal{S}^*$ is removed, and hence each of them contains some from $\mathcal{S}_{j-1}\setminus\mathcal{X}_{j-1}$, and thus contains $X$. It follows that  $X\subseteq\cap\mathcal{S}^*$. Write $S=X\cup\{p_S\}$ for $S\in\mathcal{S}^*$. Then $$M:=\cup_{S\in\mathcal{S}^*}S=X\cup\{p_S:S\in\mathcal{S}^*\}.$$  Since $\mathcal{F}$ is non-trivial, $\mathcal{F}[\overline{X}]\neq\emptyset$. For each set $F$ from the part, we have $|F\cap X|=t-1$ and $\{p_S:S\in\mathcal{S}^*\}\subseteq F$ as it $t$-intersects with every member of $\mathcal{S}^*$. Therefore,
\begin{equation}\label{equthmremoval2}
\mathcal{F}[\overline{X}]\subseteq\left\{F\in\binom{[n]}{k}:F\cap M=M\setminus\{x\}\;\mbox{for some}\;x\in X\right\}.
\end{equation}
In particular, this implies $|M|-1\leq k$, namely, $s^*:=|\mathcal{S}^*|\leq k-t+1$. 

Assume that $s^*=1$. Let $\mathcal{S}^*=\{T\}$, and fix $G\in\mathcal{F}[\overline{X}]$. Then $F\cap(G\setminus X)\neq\emptyset$ whenever $F\in\mathcal{F}[X]\setminus\mathcal{F}[T]$, and hence $\mathcal{F}[X]\setminus\mathcal{F}[T]\subseteq\bigcup_{w}\mathcal{F}[X\cup\{w\}]$ with $w$ ranging over $G\setminus T$. Note that none of those $X\cup\{w\}$ is a $t$-cover of $\mathcal{F}$. By combining this with Lemmas \ref{lemmaind} and \ref{lemmafin-stru0} (i), for all $(k-t)$-subset $P$ with $T\subseteq X\cup P$ and for any $2$-subset $I\in[n]\setminus(X\cup P)$, it holds that
\begin{align*}
	|\mathcal{F}\setminus\mathcal{K}(X,P,I)|&\leq|\mathcal{F}[\overline{T}]|\leq\left|\bigcup_{i=0}^{q-t-2}\mathcal{F}[\mathcal{X}_i]\right|+|\mathcal{F}[X]\setminus\mathcal{F}[T]|\\
	&\leq\frac{9\binom{t+2}{2}}{(1-\varepsilon_{\ref{assumption}})r_0}\binom{n-t-1}{k-t-1}+(k-t+1)^2\binom{n-t-2}{k-t-2}=\frac{O_{\eta}(k^2)}{r_0}\binom{n-t-1}{k-t-1}
\end{align*}
 Hence $|\mathcal{F}\setminus\mathcal{K}(X,P,I)|\leq2\theta|\mathcal{K}(n,k,t,2)|$ for some $\mathcal{K}(X,P,I)$ and sufficiently large $C_1$. 

Suppose $s^*\geq2$. If $\mathcal{F}[X]=\mathcal{F}[\mathcal{S}^*]$, then the two families in (\ref{equthmremoval2}) coincide as $\mathcal{F}$ is maximal, and hence  $\mathcal{F}\cong\mathcal{L}(n,k,t,s^*)$. In this case, we note that $s^*\notin\{2,k-t+1\}$ as $\mathcal{F}\not\lesssim\mathcal{A}(n,k,t)$ and  $\mathcal{F}\not\lesssim\mathcal{H}(n,k,t)$. Hence $s^*\in[3,k-t]$. Set without loss of generality that $\mathcal{F}=\mathcal{L}(n,k,t,s)$. Now it is routine to check that $\mathcal{F}\setminus\mathcal{K}(n,k,t,2)\subseteq\mathcal{F}[\overline{X}]$ and 
$$\mathcal{F}[\overline{X}]\cap\mathcal{K}(n,k,t,2)=\{([k]\setminus\{i\})\cup\{j\}:i\in[t],\;j\in\{k+1,k+2\}\}.$$
Hence
\begin{align*}
	|\mathcal{F}\setminus\mathcal{K}(n,k,t,2)|&=t\left(\binom{n-t-s}{k-t-s+1}-2\right)<t\binom{n-t-s+1}{k-t-s+1}\\
	&\leq t/r_0^{s-2}\binom{n-t-1}{k-t-1}=O_\eta(k)/r_0\binom{n-t-1}{k-t-1}, 
\end{align*}
which is less than $2\theta|\mathcal{K}(n,k,t,2)|$ for large $C_1$. 

It remains to suppose $s^*\geq2$ and  $\mathcal{F}[\mathcal{S}^*]\subsetneqq\mathcal{F}[X]$. We proceed by considering 
\begin{equation*}
	\mathcal{B}=\{F\setminus X:F\in\mathcal{F}[X]\setminus\mathcal{F}[\mathcal{S}^*]\}\;\;\mbox{and}\;\;\mathcal{C}=\{F\setminus M:F\in\mathcal{F}[\overline{X}]\}.
\end{equation*}
From (\ref{equlemmafin-stru22copy}), we have
\begin{equation*}
	\mathcal{B}\subseteq\binom{[n]\setminus M}{k-t}\;\;\mbox{and}\;\;\mathcal{C}\subseteq\binom{[n]\setminus M}{k-t-s^*+1}.
\end{equation*}
Of course $\mathcal{F}[X]\setminus\mathcal{F}[\mathcal{S}^*]=\left\{X\cup B:B\in\mathcal{B}\right\}$. 
Let $B\in\mathcal{B}$ and $C\in\mathcal{C}$, and by definition pick $F,F'\in\mathcal{F}$ with $B=F\setminus X$ and $C=F'\setminus M$. Then $F\cap F'\cap M=X\cap(F'\cap M)$ has size exactly $t-1$, and so $B\cap C=F\cap F'\cap([n]\setminus M)$ is non-empty. Hence $\mathcal{B}$ and $\mathcal{C}$ are cross-intersecting. Fix $C\in\mathcal{C}$, then $\{(M\setminus\{x\})\cup C:x\in X\}$ form a collection of $t$-covers of $\mathcal{F}$ and moreover, since $|M\setminus X|=s^*\geq2$, every two sets from this collection are $t$-intersecting. Then the maximality of $\mathcal{F}$ yields $\{(M\setminus\{x\})\cup C:x\in X\}\subseteq\mathcal{F}$. Therefore, we obtain $$\mathcal{F}[\overline{X}]=\left\{(M\setminus\{x\})\cup C:x\in X,\;C\in\mathcal{C}\right\}.$$
Combining these gives
\begin{align*}
	\mathcal{F}=&\left\{F\in\binom{[n]}{k}:S\subseteq F\;\mbox{for some}\;S\in\mathcal{S}^*\right\}\\
	&\cup\left\{X\cup B:B\in\mathcal{B}\right\}\cup\left\{(M\setminus\{x\})\cup C:x\in X,\;C\in\mathcal{C}\right\}.
\end{align*}
Further, we observe that $\cap_{C\in\mathcal{C}}C=\emptyset$. To see this, let $w\in[n]\setminus M$. Note that $X\cup\{w\}$ is not a $t$-cover of $\mathcal{F}$. Then there exists $F\in\mathcal{F}$ with $|F\cap(X\cup\{w\})|<t$, and certainly  $F\in\mathcal{F}[\overline{X}]$. Now $|F\cap X|=t-1$, and consequently $w\notin F\setminus M\in\mathcal{C}$. Thus $\cap_{C\in\mathcal{C}}C=\emptyset$, or equivalently, $\tau_1(\mathcal{C})\geq2$. Choose a $(k-t)$-subset $P$ and a $2$-subset $I$ with $M\setminus X\subseteq P$ and $I\cap(X\cup P)=\emptyset$. Let $\mathcal{K}=\mathcal{K}(X,P,I)$. Since $M\setminus X\subseteq P$, we have
$$
\left\{F\in\binom{[n]}k:X\subseteq F,\ F\cap(M\setminus X)\neq\emptyset\right\}\subseteq\mathcal{K}.
$$
It follows that $|\mathcal{F}\setminus\mathcal{K}|\leq |\mathcal B|+t|\mathcal C|$. By $\tau_1(\mathcal{C})\geq2$, $\tau_1(\mathcal{B})\leq k-t-s^*+1$, and applying Lemma \ref{lemmakey} with $t=1$, $\mathcal{F}=\mathcal{B}$, $\mathcal{G}=\mathcal{C}$ and $m=2$, we obtain that 
$$|\mathcal{B}|\leq(k-t-s^*+1)^2\binom{n-t-s^*-2}{k-t-2}.$$ 
For $|\mathcal{C}|$, we simply use the bound $(k-t)\binom{n-t-s^*-1}{k-t-s^*}$ derived from $\mathcal{C}=\bigcup_{Z\in\binom{B_0}{t}}\mathcal{C}[Z]$ for fixed $B_0\in\mathcal{B}$. Therefore, by $s^*\geq2$,
\begin{align*}
|\mathcal{F}\setminus\mathcal{K}|&\leq |\mathcal B|+t|\mathcal C|\leq(k-t-s^*+1)^2\binom{n-t-s^*-2}{k-t-2}+t(k-t)\binom{n-t-s^*-1}{k-t-s^*}\\
&<(k-t)^2\binom{n-t-4}{k-t-2}+t(k-t)\binom{n-t-3}{k-t-2}=O_\eta(k^2)/r_0\binom{n-t-1}{k-t-1}.
\end{align*}
On the other hand,  $|\mathcal{K}(n,k,t,2)|>\frac{r_0(k-t)}{r_0+(k-t)}\binom{n-t-1}{k-t-1}$. Hence $|\mathcal{F}\setminus\mathcal{K}|<2\theta|\mathcal{K}(n,k,t,2)|$ for large $C_1$. This completes the proof in Case 1.\vspace{1em}

\noindent{\bf Case 2.\;}$Ctk\log_2 k\leq n<C_1k^2$.

For simplicity, we write 
\begin{equation*}
M_0:=\binom{n-t}{k-t}-\binom{n-k}{k-t}
\end{equation*}
for short. Choose a sufficiently large $c\geq1$ such that
\begin{equation}\label{equthmremoval4}
0.2c^{-1}+2^{-c}<(1-\theta)\;\;\mbox{and}\;\;2^{-c}<\theta.
\end{equation}

Let $M=C_1(1+\eta)/\eta$, and let
$C_{\ref{thmstrongsp}}=C_{\ref{thmstrongsp}}(c,1,M)$ and $L_{\ref{thmstrongsp}}=L_{\ref{thmstrongsp}}(c,1,M)$
be the constants given in Theorem \ref{thmstrongsp}.
We choose $C\geq C_{\ref{thmstrongsp}}$ sufficiently large, as specified below.
Since $k\geq(1+\eta)t$ and $n<C_1k^2$, we have
$n<t+Mk(k-t)$.
Applying Theorem \ref{thmstrongsp} to $\mathcal{F}$ with $d=1$
gives a $t$-intersecting family
$\mathcal{S}\subseteq\binom{[n]}{\leq q}$, where
$q=t+\lceil L_{\ref{thmstrongsp}}t\log_2 k\rceil$, such that

\begin{equation}\label{equthmremoval5}
|\mathcal{F}\setminus\mathcal{F}[\mathcal{S}]|\leq k^{-ct}M_0.
\end{equation}
 We next choose $C$ sufficiently large so that
\begin{equation}\label{equlemmanearly1}
	k-t\geq240cq\;\;\mbox{and}\;\;r_0\geq\max\{240cq,4\theta^{-1}\}.
\end{equation}
Indeed, from $Ctk\log_2 k\leq n<C_1k^2$, we obtain
\begin{equation*}
	q=t+\lceil L_{\ref{thmstrongsp}}t\log_2 k\rceil\leq(L_{\ref{thmstrongsp}}+2)t\log_2 k<\frac{(L_{\ref{thmstrongsp}}+2)C_1}{C}k\leq\frac{\eta k}{240c(1+\eta)}
	\leq\frac{k-t}{240c},
\end{equation*}
where the last two inequalities follow by choosing
$C\geq240c(L_{\ref{thmstrongsp}}+2)C_1(1+\eta)/\eta$ and using
$k\geq(1+\eta)t$.
Moreover,
\begin{equation*}
r_0=\frac{n-t}{k-t}>\frac nk\geq Ct\log_2 k\geq\frac{Cq}{L_{\ref{thmstrongsp}}+2}\geq240cq
\end{equation*}
provided that $C\geq240c(L_{\ref{thmstrongsp}}+2)$.
Finally, $r_0>C\geq4\theta^{-1}$ after increasing $C$ if necessary. Note that all these requirements on $C$ depend only on $\eta$ and $\theta$.

Perform the algorithm to
$\mathcal{S}$, the uniformity $q=t+\lceil L_{\ref{thmstrongsp}}t\log_2 k\rceil$ and an arbitrary initial fingerprint 
$\mathcal{S}_0$ (the existence is ensured by Lemma \ref{lemmafingerprintdef}), and let $N$ be the number of
rounds and $(\mathcal{S}_i,\mathcal{X}_i)$ ($i\leq N$) be the output families. By Lemma \ref{lemmafingerprintproperty-single} (iv) and the estimate
for $\varphi$ in the proof of Lemma \ref{lemmarqt} (i),
for all $i\leq\min\{N,q-t-1\}$, we have
\begin{align*}
	|\mathcal{X}_i|\leq\binom{q-i}{q-t-i}(q-t-i+1)^{q-t-i}\leq\left(\frac{e(q-i)}{q-t-i}\cdot(q-t-i+1)\right)^{q-t-i}<\left(6q\right)^{q-t-i}.
\end{align*}

Suppose $N\leq q-t-1$. It follows that
\begin{align*}
	\left|\bigcup_{i=0}^{N}\mathcal{F}[\mathcal{X}_i]\right|
	&\leq
	\sum_{i=0}^{N}
	|\mathcal{X}_i|\binom{n-(q-i)}{k-(q-i)}<\sum_{i=0}^{q-t-1}\left(\frac{6q}{r_0}\right)^{q-t-i}\binom{n-t}{k-t}\\
	&<\frac{6q/r_0}{1-6q/r_0}\binom{n-t}{k-t}<\frac{6q}{r_0-6q}\cdot\frac{r_0+k-t}{k-t}M_0<0.1c^{-1}M_0,
\end{align*}
where the last inequality follows from (\ref{equlemmanearly1}). Therefore, by (\ref{equthmremoval5}) and (\ref{equthmremoval4}),
\begin{align*}
|\mathcal{F}|&\leq\left|\bigcup_{i=0}^{N}\mathcal{F}[\mathcal{X}_i]\right|+|\mathcal{F}\setminus\mathcal{F}[\mathcal{S}]|\leq(0.1c^{-1}+k^{-ct})M_0<(1-\theta)M_0
\end{align*}

Suppose $N=q-t$. Now $\mathcal{S}_{q-t}=\{X\}$ for some $t$-subset $X$. Let $j$ be the minimal index with this property.

We proceed by proving that if $j>0$, then 
$|\mathcal{F}|<(1-\theta)M_0$. To see this, note that now
\begin{equation*}
	\mathcal{F}\subseteq\left(\bigcup_{i=0}^{j-1}\mathcal{F}[\mathcal{X}_i]\right)\cup\mathcal{F}[\mathcal{S}_{j-1}\setminus\mathcal{X}_{j-1}]\cup(\mathcal{F}\setminus\mathcal{F}[\mathcal{S}]).
\end{equation*}
Set $\mathcal{H}=\mathcal{S}_{j-1}\setminus\mathcal{X}_{j-1}$ and $\mathcal{H}_{a}:=\mathcal{H}\cap\binom{[n]}{q-a}$ for $j-1\leq a\leq q-t-1$. Note that every member of $\mathcal{H}$ contains $X$ and has size
at least $t+1$; otherwise the antichain $\mathcal{S}_{j-1}$ would
equal $\{X\}$, contrary to the minimality of $j$. Then by Lemma \ref{lemmafingerprintproperty-single} (iii), we have
$$|\mathcal{H}_a|=|\mathcal{H}_a[X]|\leq(q-t-j+2)^{q-t-a}\;\;\mbox{for}\;\;j-1\leq a\leq q-t-1.$$
Therefore, 
\begin{align*}
	|\mathcal{F}[\mathcal{S}_{j-1}\setminus\mathcal{X}_{j-1}]|&\leq\sum_{a=j-1}^{q-t-1}|\mathcal{H}_a|\binom{n-(q-a)}{k-(q-a)}\leq\sum_{a=j-1}^{q-t-1}\left(\frac{q}{r_0}\right)^{q-t-a}\binom{n-t}{k-t}\\
	&<\frac{q/r_0}{1-q/r_0}\binom{n-t}{k-t}<0.1c^{-1}M_0,
\end{align*}
and we derive from (\ref{equthmremoval4}) again that
$$|\mathcal{F}|<(0.1c^{-1}+0.1c^{-1}+k^{-ct})M_0<(1-\theta)M_0.$$

Next, suppose $j=0$. Now $X\subseteq S$ for all $S\in\mathcal{S}$, and hence
\begin{equation}\label{equthmremoval3}
|\mathcal{F}[\overline{X}]|\leq|\mathcal{F}\setminus\mathcal{F}[\mathcal{S}]|\leq k^{-ct}M_0<\theta|\mathcal{K}(n,k,t,2)|
\end{equation}
as $k^{-ct}<2^{-c}<\theta$.

First, suppose that there exists $F_0\in\mathcal{F}[\overline{X}]$ with $|F_0\cap X|=:a\leq t-2$. Note that $|F_0\setminus X|=k-a\geq k-t+2$. Fix $I,P\subseteq F_0\setminus X$ with $|I|=2, |P|=k-t$ and $I\cap P=\emptyset$. Let us count the size of $\mathcal{F}\setminus\mathcal{K}(X,P,I)$. Let $F\in\mathcal{F}[X]$ with $F\cap P=\emptyset$. Then $$|F\cap F_0|=|F\cap F_0\cap X|+|F\cap(F_0\setminus(X\cup P))|\geq t.$$
It follows that $|F\cap(F_0\setminus(X\cup P))|\geq t-a$. On the other hand, note that $|F_0\setminus(X\cup P)|=k-a-(k-t)=t-a$. Hence $F_0\setminus(X\cup P)\subseteq F$ and, in particular, $I\subseteq F$. Therefore,
$$\mathcal{F}\setminus\mathcal{K}(X,P,I)\subseteq\mathcal{F}[\overline{X}],$$
and consequently $|\mathcal{F}\setminus\mathcal{K}(X,P,I)|\leq\theta|\mathcal{K}(n,k,t,2)|$ from (\ref{equthmremoval3}). 

It remains to consider the case in which $|F\cap X|=t-1$ for all $F\in\mathcal{F}[\overline{X}]$. We write
$$\mathcal{Y}=\{F\setminus X:F\in\mathcal{F}[\overline{X}]\}.$$
Note that $\mathcal{Y}\neq\emptyset$ as $\mathcal{F}$ is non-trivial. If $\mathcal{Y}$ is a singleton with the element say, $Y_0$, then
\begin{align*}
\mathcal{F}[X]&\subseteq\left\{F\in\binom{[n]}{k}:X\subseteq F,\;F\cap Y_0\neq\emptyset\right\}\;\;\mbox{and}\\
\mathcal{F}[\overline{X}]&\subseteq\{(X\setminus\{x\})\cup Y_0:x\in X\}.
\end{align*}
Thus $\mathcal{F}\lesssim\mathcal{H}(n,k,t)$, a contradiction. Hence $|\mathcal{Y}|\geq2$, and fix $Y_1\neq Y_2\in\mathcal{Y}$, $z_1\in Y_1\setminus Y_2$ and $z_2\in Y_2\setminus Y_1$. Set $P=Y_1\setminus\{z_1\}$ and $I=\{z_1,z_2\}$, and write $s=|Y_1\cap Y_2|$ for short. Note that $1\leq s\leq k-t$. Let us consider $\mathcal{F}\setminus\mathcal{K}(X,P,I)$. If $F\in\mathcal{F}[X]\setminus\mathcal{K}(X,P,I)$, then $X\subseteq F$, $F\cap P=\emptyset$ and $F\cap I=\{z_1\}$. Therefore, 
\begin{align*}
|\mathcal{F}[X]\setminus\mathcal{K}(X,P,I)|&\leq\binom{n-k-2}{k-t-1}-\binom{n-(2k-t+2-s)}{k-t-1}\\
&\leq(k-t-s)\binom{n-k-3}{k-t-2}\leq(k-t-1)\binom{n-k-3}{k-t-2}.
\end{align*}
By Lemma \ref{lemmabinom} (i), we obtain
\begin{align*}
\frac{|\mathcal{K}(X,P,I)|}{|\mathcal{F}[X]\setminus\mathcal{K}(X,P,I)|}&>\frac{\binom{n-t}{k-t}-\binom{n-k}{k-t}}{(k-t-1)\binom{n-k-3}{k-t-2}}>\frac{r_0(k-t)}{r_0+k-t}\cdot\frac{\binom{n-t-1}{k-t-1}}{(k-t-1)\binom{n-k-3}{k-t-2}}\\
&>\frac{r_0^2}{r_0+k-t}\cdot\left(1+\frac{(k-t-2)(k-t)}{n-t-2}\right)>\frac{r_0(r_0+(k-t)-2)}{r_0+(k-t)}>\theta^{-1}.
\end{align*}
The last inequality follows from (\ref{equlemmanearly1}), since
\[
\frac{r_0(r_0+k-t-2)}{r_0+k-t}
=r_0-\frac{2r_0}{r_0+k-t}
>r_0-2\geq4\theta^{-1}-2>\theta^{-1}.
\]
This together with (\ref{equthmremoval3}) yields
\begin{equation*}
|\mathcal{F}\setminus\mathcal{K}(X,P,I)|\leq|\mathcal{F}[X]\setminus\mathcal{K}(X,P,I)|+|\mathcal{F}[\overline{X}]|<2\theta|\mathcal{K}(n,k,t,2)|,
\end{equation*}
which finishes the proof.{\hfill $\square$}\vspace{1em}

\noindent{\bf Proof of Theorem \ref{corothmremoval}.}\;
Apply Theorem \ref{thmremoval} with $\theta_{\ref{thmremoval}}$ replaced by $\theta/2$ and with the corresponding constant
$C=C_{\ref{thmremoval}}(\eta,\theta/2)$. Suppose to the contrary that
$\mathcal{F}\not\lesssim\mathcal{A}(n,k,t)$, $\mathcal{F}\not\lesssim\mathcal{H}(n,k,t)$ and $
|\mathcal{F}|>(1+\theta)|\mathcal{K}(n,k,t,2)|
>(1-\theta/2)|\mathcal{K}(n,k,t,2)|$. Then 
Theorem \ref{thmremoval} yields a family
$\mathcal{K}\cong\mathcal{K}(n,k,t,2)$ such that $
|\mathcal{F}\setminus\mathcal{K}|
\leq\theta|\mathcal{K}(n,k,t,2)|$. This leads to
\[
|\mathcal{F}|
\leq|\mathcal{K}|+|\mathcal{F}\setminus\mathcal{K}|
\leq(1+\theta)|\mathcal{K}(n,k,t,2)|,
\]
a contradiction. This completes the proof.
{\hfill$\square$}
\section{Concluding Remarks}\label{secremark}
In a sequel \cite{Wen-Lv-II}, we continue our study of large
$t$-intersecting families, focusing on $t$-covers and $t$-diversity.

We conclude with two further problems concerning possible extensions
of the results proved in the present paper. A natural direction for further research is to extend the classification of large $t$-intersecting families below the range covered by Theorem \ref{thmlarge}. The problem for small $n$ (for example, $n=O(k^2)$) might be challenging. So it might be interesting to consider the following problem.
\begin{problem}
For $n=\Omega(k^3)$, classify all maximal $t$-intersecting families in $\binom{[n]}{k}$ with size at least $|\mathcal{L}(n,k,t,k-t-1)|$. 
\end{problem}

By Theorem \ref{thmlarge} and Lemma \ref{lemmacountingfamily}, we obtain Corollary \ref{coroH-K}, which provides a $t$-intersection version of Han--Kohayakawa theorem for large $n$. In particular, the lower bound on $n$ here equals $O((t+1)(k-t+1))$ for $k-t=\Theta(t)$. It would be interesting to consider whether it holds for all $k\geq t+2$.
\begin{problem}
Is there an absolute constant $C$ such that the following holds for all $k\geq t+2\geq4$ and $n\geq Ckt$? If  $\mathcal{F}$ is a non-trivial $t$-intersecting family which is isomorphic to neither a subfamily of $\mathcal{A}(n,k,t)$ nor a subfamily of $\mathcal{H}(n,k,t)$, then $|\mathcal{F}|\leq|\mathcal{K}(n,k,t,2)|$. 
\end{problem}
\section*{Acknowledgments}
B. Lv is supported by National Natural Science Foundation of China (12571347 \& 12131011), and Beijing Natural Science Foundation (1252010).
\
\addcontentsline{toc}{chapter}{Bibliography}

{
	}

\begin{thebibliography}{99}
				\setlength{\itemsep}{-1pt}
		\bibitem{Ahlswede-Khachatrian-1996}
		R. Ahlswede and L.H. Khachatrian, The complete nontrivial-intersection theorem for systems of finite sets, J. Combin. Theory Ser. A 76 (1996) 121--138.
		
		\bibitem{Ahlswede-Khachatrian-1997}
		R. Ahlswede and L.H. Khachatrian, The complete intersection theorem for systems of finite sets, European J. Combin. 18 (1997) 125--136.
		
		\bibitem{sunflower} R. Alweiss, S. Lovett, K. Wu and J. Zhang, Improved bounds for the sunflower lemma, Ann. of Math. 194 (3) (2021) 795--815.
		
		\bibitem{Lv-2021}
		M. Cao, B. Lv and K. Wang, The structure of large non-trivial $t$-intersecting families of
		finite sets, European J. Combin. 97 (2021) 103373.
		
		\bibitem{Cao-Lu-Lv-Wang-2024} M. Cao, M. Lu, B. Lv and K. Wang, Nearly extremal non-trivial cross $t$-intersecting families and $r$-wise $t$-intersecting families, European J. Combin. 120 (2024) 103958.
		
		\bibitem{Dinur-Friedgut}I. Dinur and E. Friedgut,  Intersecting families are essentially contained in juntas, 
		Combin. Probab. Comput. 18 (2009) 107--122.
		
		\bibitem{Ellis-2019} D. Ellis, N. Keller and N. Lifshitz, Stability versions of  Erd\H{o}s--Ko--Rado type theorems via isoperimetry, J. Eur. Math. Soc. 21 (2019) 3857--3902.
		
		\bibitem{Ellis-book} D. Ellis, Intersection problems in extremal combinatorics: theorems, techniques and questions 
		old and new, in: Surveys in Combinatorics 2022, in: London Math. Soc. Lecture Note Ser., vol. 481, 
		Cambridge Univ. Press, Cambridge, 2022, pp. 115--173.
		
		\bibitem{Ellis-2024} D. Ellis, N. Keller and N. Lifshitz, Stability for the complete intersection theorem, and the
		forbidden intersection problem of Erd\H{o}s and S\'{o}s, J. Eur. Math. Soc. 26 (2024)  1611--1654.
		
		\bibitem{Erdos-Ko-Rado-1961}
		P. Erd\H{o}s, C. Ko and R. Rado, Intersection theorems for systems of finite sets, Quart. J. Math. Oxf. 2 (12) (1961) 313--320.
		
		\bibitem{Frankl-1976}
		P. Frankl, The Erd\H{o}s-Ko-Rado theorem is true for $n = ckt$, in: Combinatorics, Vol. I, Proc. Fifth Hungarian Colloq., Keszthely, 1976, in: Colloq. Math. Soc. J\'{a}nos Bolyai, vol. 18, North-Holland, 1978, 365--375.
		
		\bibitem{Frankl-1978}
		P. Frankl, On intersecting families of finite sets, J. Combin. Theory Ser. A 24 (1978) 146--161.
		%
		\bibitem{Frankl-1987} P. Frankl, Erd\H{o}s--Ko--Rado theorem with conditions on the maximal degree, J. Combin. Theory Ser. A 46 (1987) 252--263.
		
		\bibitem{Frankl-Tokushige-2016}
		P. Frankl and N. Tokushige, Invitation to intersection problems for finite sets, J.
		Combin. Theory Ser. A 144 (2016) 157--211.
		
		\bibitem{Frankl-2020} P. Frankl, Maximum degree and diversity in intersecting hypergraphs, J. Combin. Theory Ser. B 144 (2020) 81--94.
		
		\bibitem{Frankl-Kupavskii-2021} P. Frankl and A. Kupavskii, Diversity, J. Combin. Theory Ser. A 182 (2021) 105468.	
		
		\bibitem{Frankl-Kupavskii-2025}P. Frankl and A. Kupavskii, The Hajnal--Rothschild problem,  arXiv:2502.06699.
			
		\bibitem{Frankl-2025} P. Frankl, Concise proofs concerning the size and structure of large intersecting $k$-graphs, Acta Math. Hungar. 177 (2025) 247--264.
		
		\bibitem{Ge-Xu-Zhao}G. Ge, Z. Xu and X. Zhao, Algebraic approach to stability results for Erd\H{o}s--Ko--Rado theorem, arXiv:2410.22676.
		
		\bibitem{H-R}A. Hajnal and B. Rothschild, A Generalization of the Erd\H{o}s--Ko--Rado Theorem on Finite Set Systems, J. Combin. Theory Ser. A 15 (1973), 359--362.
		
		\bibitem{Han-Kohayakawa}
		J. Han and Y. Kohayakawa, The maximum size of a non-trivial intersecting uniform family that is not a subfamily of the Hilton-Milner family, Proc. Amer. Math. Soc. 145(1) (2017) 73--87.
		
		\bibitem{Hilton-Milner-1967}
		A.J.W. Hilton and E.C. Milner, Some intersection theorems for systems of finite sets, Quart. J. Math. Oxf. 2 (18) (1967) 369--384.
		
		\bibitem{Huang-Kupavskii-2026} Y. Huang and A. Kupavskii, Structure and properties of large cross-intersecting families, arXiv:2606.20085.
		
		 \bibitem{Huang-Peng}
		Y. Huang and Y. Peng, Stability of intersecting families, European J. Combin. 115 (2024) 
		103774.
		
		\bibitem{Ihringer-Kupavskii} F. Ihringer and A. Kupavskii, Structure of $t$-intersecting families of vector spaces, arXiv:2605.02698.
		
		\bibitem{Keevash} P. Keevash, Shadows and intersections: Stability and new proofs, Adv. Math. 218 (2008) 1685--1703.
		
		\bibitem{Keevash-Long-2020} P. Keevash and E. Long, Stability for vertex isoperimetry in the cube, J. Combin. Theory Ser. B 145 (2020) 113--144.
		
		\bibitem{Keller-Kupavskii-Lifshitz-Sheinfeld} N. Keller, A. Kupavskii, N. Lifshitz and O. Sheinfeld, A complete intersection theorem for large permutation groups, 	arXiv:2607.00318.
		
		\bibitem{Kostochka-Mubayi}
		A. Kostochka and D. Mubayi, The structure of large intersecting families, Proc. Amer. Math. Soc. 145 (2017) 2311--2321.
		
		\bibitem{Kupavskii-2018}
		A. Kupavskii, Structure and properties of large intersecting families, arXiv:1810.00920.
		
		\bibitem{Kupavskii-2023} A. Kupavskii, Intersection theorems for uniform subfamilies of hereditary families, arXiv:2311.02246.
			
		\bibitem{Kupavskii-2024-perm}A. Kupavskii, An almost complete $t$-intersection theorem for permutations, arXiv:2405.07843.
			
		\bibitem{Kupavskii-2025}
		A. Kupavskii, Structure of non-trivial intersecting families, Proc. Amer. Math. Soc. 153 (2025) 2863--2873. 
		
		\bibitem{Kupavskii-2026}
		A. Kupavskii, Erd\H{o}s--Ko--Rado type results for partitions via spread approximations, European J. Combin. 132 (2026) 104288.
		
		\bibitem{Kupavskii-Zakharov-2018} A. Kupavskii and D. Zakharov, Regular bipartite graphs and intersecting families, J. Combin. Theory Ser. A 155 (2018) 180--189.
		
		\bibitem{Kupavskii-2024}
		A. Kupavskii and D. Zakharov, Spread approximations for forbidden intersections problems, Adv. Math. 445 (2024) 109653. 

		\bibitem{Lemons-Palmer} N. Lemons and C. Palmer, The unbalance of set systems, Graphs  Combin. 24 (2008) 361--365.
		
		\bibitem{Stoeckl} M. Stoeckl, Lecture notes on recent improvements for the sunflower lemma, 2022, \url{https://mstoeckl.com/notes/research/sunflower_notes.html}.
		
		\bibitem{Tao} T. Tao, The sunflower lemma via Shannon entropy, 2020, \url{https://terrytao.wordpress.com/2020/07/20/the-sunflower-lemma-via-shannon-entropy/}.
		
		\bibitem{Wen-Lv-2026+}J. Wen and B. Lv, A unified approach to cross-intersection problems with applications to Hilton--Milner type theorems and stability, arXiv:2607.03315.
		
		\bibitem{Wen-Lv-II}
		J. Wen and B. Lv, Structure of large $t$-intersecting families II: $t$-covers and $t$-diversity, in preparation.
		
		\bibitem{Wilson-1984}
		R.M. Wilson, The exact bound in the Erd\H{o}s-Ko-Rado theorem, Combinatorica 4 (1984) 247--257.
		
		\bibitem{Wu et al.}
		Y. Wu, Y. Li, L. Feng, J. Liu and G. Yu, Maximal intersecting families revisited, Discrete Math. 349 (2026) 114654.
		
		\bibitem{Yao-Liu-Wang} T. Yao, D. Liu and K. Wang, Extremal $t$-intersecting families for finite sets with
		$t$-covering number at least $t+2$, arXiv:2602.14129.
		
\end{thebibliography}
\end{document}